\documentclass[12pt,verbatim]{amsart}
\usepackage{amsfonts}
\usepackage{amsthm}
\usepackage{amsmath}
\usepackage{amssymb}
\usepackage{amscd}
\usepackage{subfigure}

\usepackage{ifthen}
\usepackage{graphicx}
\usepackage{caption}

\newcounter{minutes}
\divide\time by 60
\newcounter{hours}
\multiply\time by 60 \addtocounter{minutes}{-\time}

\theoremstyle{plain}
\newtheorem{thm}[equation]{Theorem}
\newtheorem{cor}[equation]{Corollary}
\newtheorem{lem}[equation]{Lemma}

\newtheorem{prop}[equation]{Proposition}
\newtheorem{conjecture}[equation]{Conjecture}

\theoremstyle{definition}
\newtheorem{defn}[equation]{Definition}

\theoremstyle{remark}
\newtheorem{rem}[equation]{Remark}

\newtheorem{nonsec}[equation]{}

\newcommand{\R}{{\mathbb R}}
\newcommand{\Rn}{{\R}^n}

\newcommand{\BB}{{\mathbb B}^2}
\newcommand{\Bn}{{\mathbb B}^n}
\newcommand{\UH}{\mathbb{H}^2}
\newcommand{\Hn}{{\mathbb H}^n}

\newcommand{\bdr}{\partial}
\newcommand{\ang}{\measuredangle}
\def\be{\begin{equation}}
\def\ee{\end{equation}}

\def\tang{\mathrel{\hbox{\rlap{%
\kern 0.2ex\hbox to 1.8ex{\hrulefill}}\raise2.6pt\hbox{$\bigcirc$}}}}

\DeclareMathOperator{\ray}{ray}
\DeclareMathOperator{\card}{card}

\DeclareMathOperator{\comp}{comp}
\DeclareMathOperator{\re}{Re}
\DeclareMathOperator{\im}{Im}

\numberwithin{equation}{section}

\begin{document}







\def\thefootnote{}
\footnotetext{ \texttt{\tiny File:~\jobname .tex,
          printed: \number\year-\number\month-\number\day,
          \thehours.\ifnum\theminutes<10{0}\fi\theminutes}
} \makeatletter\def\thefootnote{\@arabic\c@footnote}\makeatother

\title{the visual angle metric and M\"obius transformations}

\author{Riku Kl\'en}
\author{Henri Lind\'en}
\author{Matti Vuorinen}
\author{Gendi Wang}

\address{Department of Mathematics and Statistics, University of Turku, Turku 20014, Finland}
\email{riku.klen@utu.fi}\email{henri.linden@iki.fi} \email{vuorinen@utu.fi} \email{genwan@utu.fi}

\maketitle

\begin{abstract}
A new similarity invariant metric $v_G$ is introduced. The visual angle metric $v_G$ is defined on a domain $G\subsetneq\Rn$ whose boundary is not a proper subset of a line. We find sharp bounds for $v_G$ in terms of the hyperbolic metric in the particular case when the domain is either the unit ball $\Bn$ or the upper half space $\Hn$. We also obtain the sharp Lipschitz constant for a M\"obius transformation $f: G\rightarrow G'$ between domains $G$ and $G'$ in $\Rn$
with respect to the metrics $v_G$ and $v_{G'}$. For instance, in the case $G=G'=\Bn$ the result is sharp.
\end{abstract}



\section{Introduction}

As the name suggests, metrics play a crucial role in geometric
function theory of the plane. The three classical geometries,
the Euclidean, the  hyperbolic, and the spherical geometry each has a metric,
the Euclidean, the hyperbolic and the spherical metric, respectively, that is
invariant under a group of M\"obius transformations. These three groups of transformations are isometric automorphisms of the
respective spaces, the
complex plane ${\mathbb C}\,,$ the unit disk, and the extended
complex plane
 $\overline{\mathbb C} = {\mathbb C} \cup \{\infty \}\,.$

In addition to these classical geometries, many novel ways to look
at geometric function theory have been introduced recently, in the study of
hyperbolic type geometries. For instance, the Apollonian metric, the M\"obius
invariant metric, the quasihyperbolic metric, and some weak metrics have been
studied in \cite{ast, be2, himps, kl, l2, pt, se}. Furthermore, metrics
based on conformal capacity have been studied in \cite{se}.
On one hand, these metrics share several properties of the hyperbolic metric and are therefore
sometimes called hyperbolic type metrics. On the other hand, these metrics differ from the
classical hyperbolic metric in other respects. This circumstance suggests a wide spectrum of open
problems concerning geometries of these metric spaces and homeomorphisms between two
such spaces. Several open problems of this character were listed in \cite{vu2}. While some of these
problems have been studied and solved \cite{k, kvz, l2, se}, the systematic study of these problems is still
in its initial stages.

The metrics introduced in \cite{ast} provide a way to connect the behavior of the basic notion of an angle to the behavior of a map.
Classically one studies angle distortion locally, "in the small", whereas in
\cite{ast} this topic is studied "in the large". Here we give an
alternative way to look at this topic and introduce here what
we call the visual angle metric$^*$.
We begin by the formulation
of some of our main results. For the definitions, the reader is referred to Section 2.

\def\thefootnote{}
\footnotetext{$^*$The term "visual metric" occurs in a different meaning in the study of Gromov hyperbolic spaces, see \cite[section 3.3]{bs}.}

\begin{thm}\label{mthm1}
For $G\in\{\Bn,\Hn\}$ and $x\,,y\in G$, let $\rho_{G}^*(x,y)=\arctan\left({\rm sh}\frac{\rho_G(x,y)}{2}\right)$. Then
$$\rho_{G}^*(x,y)\leq v_G(x,y)\leq 2\rho_{G}^*(x,y).$$
The left-hand side of the inequality is sharp and the constant 2 in the right-hand side of the inequality is the best possible.
\end{thm}

The conditions for the equality case between the visual angle metric and the hyperbolic metric both in the unit ball and in the upper half space are shown in Lemma \ref{mthmble} and Lemma \ref{mthmhle}, respectively.

\begin{thm}\label{vmthm1}
Let $f: \Bn\rightarrow \Bn$ be a M\"obius transformation. Then
$$\sup_{f\in \mathcal{GM}(\Bn),\atop x\neq y\in \Bn}\frac{v_{\Bn}(f(x),f(y))}{v_{\Bn}(x,y)}=2.$$
\end{thm}

\begin{thm}\label{vmthm2}
Let $f: \UH\rightarrow \BB=f\UH$ be a M\"obius transformation. Then for all $x\,,y\in\UH$
$$v_{\UH}(x,y)/2\le v_{\BB}(f(x),(y))\le 2 v_{\UH}(x,y),$$
and the constants $1/2$ and $2$ are both the best possible.
\end{thm}

\begin{thm}\label{vmthm3}
Let $a\,,b\,,c\,,d\in\R$ and $ad-bc=1$ and $c\neq 0$. Let $f: \UH\rightarrow \UH$ be a M\"obius transformation with $f(z)=\frac{az+b}{cz+d}$. Then
$$\sup_{x\neq y\in\UH}\frac{v_{\UH}(f(x),f(y))}{ v_{\UH}(x,y)}=2.$$
\end{thm}

\bigskip

\section{Definitions and preliminary results}
\begin{nonsec}{\bf Notation.}
Throughout this paper we will discuss domains $G \subset \mathbb{R}^n$, i.e.,
open and connected subsets of $\mathbb{R}^n$. For $x,y \in G$ the Euclidean distance between
$x$ and $y$ is denoted by $|x-y|$ or $d(x,y)$, as usual. The notation
$d(x,\bdr G)$ stands for the distance from the point $x$ to the boundary $\bdr G$ of the domain $G$.
The Euclidean $n$-dimensional
ball with center $z$ and radius $r$ is denoted by $\Bn(z,r)$, and its boundary
sphere by $S^{n-1}(z,r)$. In particular, $\Bn(r)=\Bn(0,r),\;
S^{n-1}(r)=S^{n-1}(0,r)$, and $\Bn = \Bn(0,1),\; S^{n-1}=S^{n-1}(0,1)$.
The upper Lobachevsky $n$-dimensional half space (as a set) is denoted
by $\Hn = \{(z_1,z_2,\cdots, z_n) \in \Rn \,: \,  z_n > 0\}$. For $t\in \R$ and $a\in\Rn\setminus\{0\}$ we denote a hyperplane in $\overline{\Rn}=\Rn\cup\{\infty\}$ by $P(a,t)=\{x\in\Rn: x\cdot a=t\}\cup\{\infty\}$.

Given two points $x$ and $y$, the segment between them is denoted by
$$
[x,y]=\{(1-t)x+ty\;:\; 0\le t\le1\}.
$$
Given
a vector $u\in\Rn\setminus\{0\}$ and a point $x \in \Rn$, the line passing through
$x$ with direction vector $u$ is denoted by $L(x,u)$. The open ray
emanating from $x$ in the direction of $u$ is denoted by $\ray (x,u)$.
The hyperplane orthogonal to $u$ and passing through $x$ is denoted by
$P_x(u)$. Hence
\begin{eqnarray*}
L(x,u)&=&\{x+tu\ :\ t \in \R\},\\
\ray(x,u)&=&\{x+tu\ :\ t>0\},\\
P_x(u)&=&P(u,x\cdot u).
\end{eqnarray*}
Given three distinct points $x\,,y\,,z \in \Rn$, the notation $\ang(x,z,y)$
means the angle in the range $[0,\pi]$ between the
segments $[x,z]$ and $[y,z]$.
\end{nonsec}
\begin{nonsec} {\bf M\"obius transformations.}
The group of M\"obius transformations in $\overline{\Rn}$ is generated by transformations of two types:

(1) reflections in $P(a, t)$

$$f_1(x)=x-2(x\cdot a-t)\frac{a}{|a|^2},\,\, f_1(\infty)=\infty,$$
where $a\in\Rn\setminus\{0\}$ and $t\in\R$;

(2) inversions (reflections) in $S^{n-1}(a,r)$

$$f_2(x)=a+\frac{r^2(x-a)}{|x-a|^2},\,\,f_2(a)=\infty, f_2(\infty)=a,$$
where $a\in\Rn$ and $r>0$.
If $G\subset\overline{\Rn}$ we denote by $\mathcal{GM}(G)$ the group of all M\"obius transformations which map $G$ onto itself.

We denote $a^*=\frac{a}{|a|^2}$ for $a\in\Rn\setminus\{0\}$, and $0^*=\infty$, $\infty^*=0$. For a fixed $a\in\Bn\setminus\{0\}$, let
$$
\sigma_a(z)=a^*+r^2(x-a^*)^*,\,\,r^2=|a|^{-2}-1
$$
be the inversion in the sphere $S^{n-1}(a^*,r)$ orthogonal to $S^{n-1}$. Then $\sigma_a(a)=0$,  $\sigma_a(a^*)=\infty$.

Let $p_a$ denote the reflection in the $(n-1)$-dimensional hyperplane $P(a,0)$ and define a sense-preserving M\"obius transformation by
\be\label{ta}
T_a=p_a\circ \sigma_a.
\ee
Then $T_a\Bn=\Bn$, $T_a(a)=0$, and $T_a(e_a)=e_a$, $T_a(-e_a)=-e_a$. For $a=0$ we set $T_0=id$, where $id$ stands for the identity map. It is easy to see that $(T_a)^{-1}=T_{-a}$. It is well-known  that
there is an orthogonal map $k$ such that $g=k\circ T_a$ if $g\in\mathcal{GM}(\Bn)$, where $a=g^{-1}(0)$  \cite[Theorem 3.5.1]{be1}.

\end{nonsec}

\begin{nonsec} {\bf Lipschitz mappings.}
 Let $(X,d_X)$ and $(Y,d_Y)$ be metric spaces. Let $f: X\rightarrow Y$ be continuous and let $L\geq1$. We say that $f$ is $L$-lipschitz if
\begin{eqnarray*}
d_Y(f(x),f(y))\leq L d_X(x,y),\,\,{\rm for}\, x,\,y\in X,
\end{eqnarray*}
and $L$-bilipschitz if $f$ is
a homeomorphism and
\begin{eqnarray*}
d_X(x,y)/L\leq d_Y(f(x),f(y))\leq L d_X(x,y),\,\,{\rm for} \,x,\,y\in X.
\end{eqnarray*}
A $1$-bilipschitz mapping is called an isometry.

\end{nonsec}

\begin{nonsec} {\bf Absolute Ratio.}
For an ordered quadruple $a,b,c,d$ of distinct points in $\overline{\Rn}$ we define the absolute ratio by
$$|a,b,c,d|=\frac{q(a,c)q(b,d)}{q(a,b)q(c,d)},$$
where $q(x,y)$ is the chordal metric,  defined by
\begin{eqnarray*}
\left\{\begin{array}{ll}
q(x,y)=\frac{|x-y|}{\sqrt{1+|x|^2}\sqrt{1+|y|^2}},&\,\,\, x\,,y\neq\infty,\\
q(x,\infty)=\frac{1}{\sqrt{1+|x|^2}},&\,\,\, x\neq\infty,
\end{array}\right.
\end{eqnarray*}
for $x\,,y\in\overline{\Rn}$.
Note also that for distinct points $a,b,c,d\in \Rn$
$$
|a,b,c,d|=\frac{|a-c||b-d|}{|a-b||c-d|}.
$$
The most important property of the absolute ratio is the M\"obius invariance, see \cite[Theorem 3.2.7]{be1}, i.e., if $f$ is a M\"obius transformation, then
$$|f(a),f(b),f(c),f(d)|=|a,b,c,d|,$$
for all distinct $a,b,c,d\in\overline{\Rn}$.

\end{nonsec}

\begin{nonsec}{\bf Hyperbolic metric.}
By \cite[p.35]{be1} we have
\be\label{cosh}
{\rm ch}\rho_{\Hn}(x,y)=1+\frac{|x-y|^2}{2 x_n y_n}
\ee
for all $x,y\in \Hn$, and by \cite[p.40]{be1} we have
\be\label{sinh}
{\rm sh}\frac{\rho_{\Bn}(x,y)}{2}=\frac{|x-y|}{\sqrt{1-|x|^2}\sqrt{1-|y|^2}}
\ee
for all $x,y\in \Bn$. In particular, for $t\in(0,1)$,
\be\label{arth}
\rho_{\BB}(0,t e_1)=\log\frac{1+t}{1-t}=2{\rm arth} t.
\ee
The hyperbolic metric is invariant under M\"obius transformations of $G$ onto $G'$ for $G\,,G'\in\{\Hn, \Bn\}$.

\end{nonsec}

\begin{nonsec}{\bf Distance ratio metric.}
For a proper open subset $G \subset {\mathbb R}^n\,$ and for all
$x,y\in G$, the  distance ratio
metric $j_G$ is defined as
\begin{eqnarray*}
 j_G(x,y)=\log \left( 1+\frac{|x-y|}{\min \{d(x,\partial G),d(y, \partial G) \} } \right)\,.
\end{eqnarray*}
The distance ratio metric was introduced by F.W. Gehring and B.P. Palka
\cite{gp} and in the above simplified form by M. Vuorinen \cite{vu1}. Both definitions are
frequently used in the study of hyperbolic type metrics \cite{himps}
and geometric theory of functions.

\end{nonsec}

\begin{nonsec}{\bf Quasihyperbolic metric.}
Let $G$ be a proper subdomain of ${\mathbb R}^n\,$. For all $x,\,y\in G$, the quasihyperbolic metric $k_G$ is defined as
$$k_G(x,y)=\inf_{\gamma}\int_{\gamma}\frac{1}{d(z,\partial G)}|dz|,$$
where the infimum is taken over all rectifiable arcs $\gamma$ joining $x$ to $y$ in $G$ \cite{gp}.

\end{nonsec}

\begin{nonsec}{\bf Ptolemaic angular metrics.}
We start by defining the two so called Ptolemaic metrics, the first of which
has recently been studied in \cite{ast}. The other one is a one-boundary-point
version of the same metric.

We call a metric space $(X,m)$ a {\it Ptolemaic space} if it satisfies the {\it Ptolemy inequality}
$$
  m(x,z)m(y,w) \le m(x,w)m(y,z)+m(x,y)m(w,z)
$$
for all $x,y,z,w \in X$. In general, metric spaces need not be Ptolemaic, but for instance, all inner
product spaces are, and in particular, the Euclidean space is Ptolemaic. Also
note that if we choose, for instance, $w=\infty$, then Ptolemy's inequality in the Euclidean space reduces to
the triangle inequality
$$|x-z| \le |y-z|+|x-y|.$$
Now define the {\it angular characteristic} of four points $a,b,c,d$ by
$$
\sigma(a,b,c,d)=\frac{|a-c||b-d|}{|a-b||c-d|+|a-d||b-c|}.
$$
V.V. Aseev, A.V. Sych\"ev, and A.V. Tetenov proved in \cite[Lemma 1.6]{ast}, that for a
given Ptolemaic space $X$, nonempty sets $A,B \subset X$ for which $\card(A \cup B)
\ge 2$ and $A \cup B \ne X$, the function
$$r_{AB}(x,y) = \sup_{a \in A, b\in B,a\ne b} \sigma(a,x,b,y)$$
is a metric on $X \setminus (A \cup B)$. They called this the {\it angular metric},
and also proved its M\"obius invariance in \cite[Theorem 2.4]{ast}.
In this article, we
consider also the word angular metrics in different meanings. However, applying the work of Aseev and his collaborators
to the case $X=\overline{G},\; A=B=\bdr G$, where $G \subsetneq\Rn$
is a domain, we obtain the following definition:

\begin{defn}\label{propr}
Given $G \subsetneq \Rn$ and $x,y \in G$, we define
a M\"obius invariant metric by
$$r_G(x,y)= \sup_{z,w \in \bdr G, z\ne w} \sigma(z,x,w,y)=\sup_{z,w \in \bdr G, z\ne w}
\frac{|z-w||x-y|}{|z-x||w-y|+|z-y||w-x|},$$
and call this the {\it Ptolemaic angular metric}.
\end{defn}

As the reader may have noticed, for instance, the Apollonian and
half-Apollonian metrics are actually defined in the same way, only in the
half-Apollonian case one of the boundary points in the definition is
``forced'' to infinity. This approach is used in other related metrics as
well, such as the frequently used distance ratio metrics $j$ and $\tilde{j}$
by M.~Vuorinen and F.W. Gehring, respectively (see \cite{vu1} for properties of
these metrics). Lately, A.~Papadopoulos and M.~Troyanov have described this
as such metrics being two different symmetrizations, the max-symmetrization
and the mean value-symmetrization, of the same weak metric, see \cite{pt}.

We next develop the one-point version of the Ptolemaic angular metric. This
was also briefly considered in \cite[p.192]{ast} and \cite[Lemma 6.1]{ha}.
However, to our knowledge this particular metric has not been studied to any
further extent.

\begin{defn} \label{triprop}
Given $G \subsetneq \Rn$ and $x,y \in G$, we define
a similarity invariant metric by
$$
  s_G(x,y)= \sup_{z \in \bdr G} \frac{|x-y|}{|z-x|+|z-y|} \in [0,1],
$$
and call this the {\it triangular ratio metric}.
\end{defn}

The mutual order between the two metrics $s_G$ and $r_G$ is a direct
consequence of the definitions.

\begin{prop}\label{slessr}
Let $G \subsetneq \Rn$ and $\infty\in\partial G$, then for all $x,y \in G$
$$
  s_G(x,y) \le r_G(x,y).
$$
\end{prop}
\begin{proof}
By Definition \ref{propr} and Definition \ref{triprop},
we get
\begin{eqnarray*}
s_G(x,y)&=&\sup_{z \in \partial G} \sigma(z,x,\infty,y)
\le \sup_{z,w \in \partial G} \sigma(z,x,w,y) = r_G(x,y).
\end{eqnarray*}
\end{proof}
\begin{rem} \label{envs}
Even if perhaps the most natural way to define the triangular ratio
metric is the direct formula given in Definition \ref{triprop}, it is also
possible to give a definition similar to the one for the visual angle metric as follows.
Namely, given two distinct points $x,y \in \Rn$ and $c > 0$, the set
$$\{z \in \Rn\; :\; |x-z|+|y-z|=c\}$$
is known to be an ellipsoid with $x$ and $y$ as foci, semimajor axis $c/2$ and
$n-1$ semiminor axes $\sqrt{c^2-|x-y|^2}/2$. Denoting the ellipsoid
with foci $x$ and $y$, semimajor axis $b$ and $n-1$ semiminor axes $a$ by
$F(x,y;a,b)$, we may define the $c-${\it envelope} of the pair $(x,y)$ as
$$F_{xy}^c=[x,y] \cup \left(\bigcup_{|x-y|<t\leq c} F(x,y;{\textstyle \frac{1}{2}}\sqrt{t^2-|x-y|^2},t/2)\right).$$
Then it is easy to see that the above set has an alternative definition,
namely,
$$
  F_{xy}^c = \left\{z \in \Rn\; : \; |x-z|+|y-z| \le c \right\}.
$$
\end{rem}

Therefore, we can define the triangular ratio metric in another way.

\begin{rem}\label{fs}
For all $x,y \in G$, the metric $s_G$ is defined by
$$
  s_G(x,y) = \frac{|x-y|}{\inf \{ c\; :\; F_{xy}^c \cap \partial G \ne \emptyset \} }.
$$

This geometric approach to the triangular ratio metric gives a convenient
tool, for instance, to compare it with the visual angle metric, as will be seen
later.
\end{rem}

\end{nonsec}

\begin{nonsec}{\bf Visual angle metric.}
We introduce two versions of angle metrics, a "one-point version" corresponding to a max-argument
and a "two-point version" corresponding to a mean-value argument.
The one-point version of this metric was introduced for dimension
$n=2$ in \cite{l1}, there also under the name "angular metric", and
the notation $\omega_G$.

We begin by introducing some geometric concepts and notation. Let
$x,y \in \Rn$, $x \ne y$ and $0< \alpha < \pi$. Let $m=(x+y)/2$ be the midpoint of the segment $[x,y]$ and define $P_{xy}=P_m(x-y)$. Let $C(x, y, z)$ be the circle centered at $z\in P_{xy}$ containing $x$ and $y$. More precisely, if $z\neq m$, then
$C(x,y,z)=S^{n-1}(z,r)\cap \Pi_{xyz}$, where $r=|z-x|=|z-y|$ and $\Pi_{xyz}$ stands for the plane passing through $x,y,z$; if $z=m$, then $C(x,y,z)$ is an arbitrary circle with diameter $[x,y]$.

Now denote
$$  {\mathcal C}_{xy}^{\alpha} =\left\{C(x,y,z) \; :\; z \in P_{xy},\; 2|z-x|\sin\alpha =|x-y|\right\}.$$
Every circle $C \in {\mathcal C}_{xy}^\alpha$ contains the points $x$ and $y$ and therefore $C \setminus \{x,y\}$ consists of two circular arcs. We denote these two circular arcs by $\comp_{\alpha}(C)$ and
$\comp_{\pi-\alpha}(C)$ and assume that the length of $\comp_{\alpha}(C)$ is equal to $2(\pi-\alpha)|x-z|$, see Figure \ref{omegafig}. Then it is clear that
$$
  C=\{x\} \cup \{y\} \cup \comp_{\alpha}(C) \cup \comp_{\pi-\alpha}(C).
$$

\begin{figure}[ht]
  \begin{center}
    \includegraphics[height=5cm]{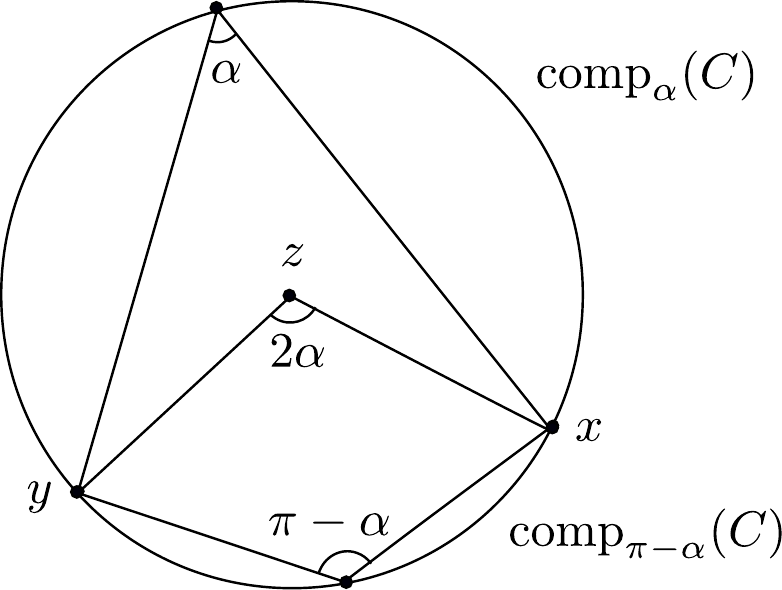}
  \end{center}
  \caption{Components $\comp_{\alpha}(C)$ and $\comp_{\pi-\alpha}(C)$ of the circle $C$.\label{omegafig}}
\end{figure}

Finally, we define the {\it $\alpha$-envelope} of the pair $(x,y)$ to be
$$
  E_{xy}^{\alpha} = [x,y] \cup \left(\bigcup_{C \in {\mathcal C}_{xy}^{t},\; \alpha\le t < \pi}\comp_{t}(C) \right)
$$
if $0<\alpha<\pi$, $E_{xy}^0=\Rn$,
and $E_{xy}^{\pi}=[x,y]$. For instance, in the case $n=3$, this means that
for $0<\alpha < \pi/2$, the set $E_{xy}^{\alpha}$ is an "apple domain"; for $\alpha = \pi/2$, the closed ball $\overline{\mathbb{B}^n(m, |x-y|/2)}$; and for $\pi/2<\alpha <\pi$, a "lemon domain".

The set $E_{xy}^{\alpha}$ has the property that for all
$w \in \bdr E_{xy}^{\alpha}$ the angle $\ang(x,w,y)$ equals to $\alpha$.

\begin{rem}\label{whatisE}
It is not difficult to show that in fact
$$E_{xy}^{\alpha} = \{ w \in \Rn \;: \; \ang(x,w,y) \ge \alpha \}.$$
\end{rem}

It is easy to see that (Figure \ref{angmetfig2})
\begin{eqnarray}\label{ef}
E_{xy}^{2\arcsin\frac{|x-y|}{c}}\subset F_{xy}^c\,\,\,\,{\rm and}\,\,\,\,E_{xy}^\alpha\subset F_{xy}^{|x-y|/\sin\frac{\alpha}{2}}.
\end{eqnarray}

\begin{figure}[h]
    \includegraphics[height=4cm]{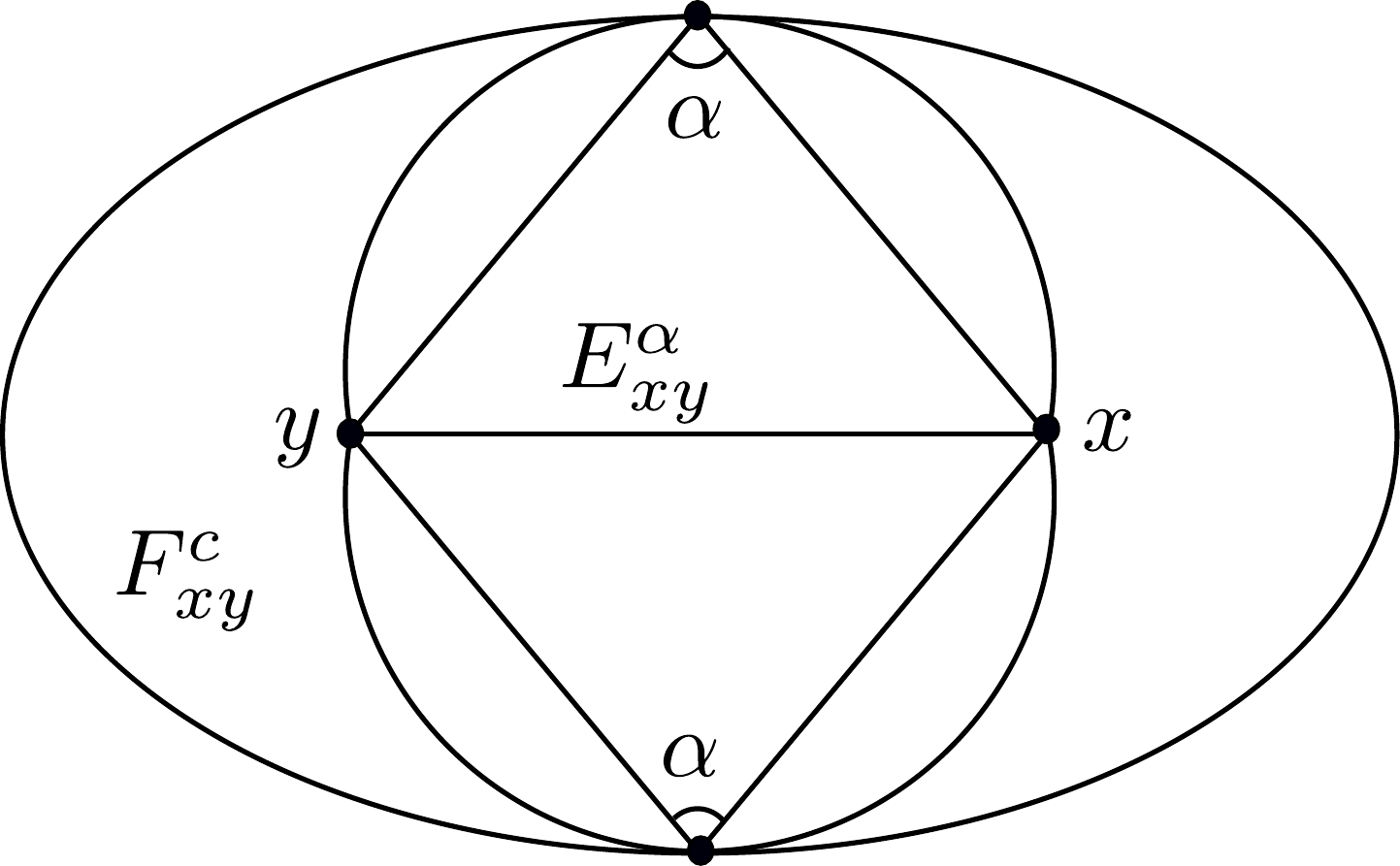}
  \caption{Here $\alpha=2\arcsin\frac{|x-y|}{c}\,\,\,\,{\rm or}\,\,\,\,c=|x-y|/\sin\frac{\alpha}{2}.$ \label{angmetfig2}}
\end{figure}

Now we are ready for the following definition:

\begin{defn}\label{maxang}
Let $G \subsetneq \Rn$ be a domain and $x,y \in G$. We define a distance function
$v_G$ by
\begin{eqnarray*}
v_G(x,y) &=& \sup \left\{\alpha\; :\; E_{xy}^\alpha \cap \bdr G \ne \emptyset  \right\}.
\end{eqnarray*}
\end{defn}

\medskip

Next we show that the function $v_G$ in fact defines a similarity
invariant metric in all domains where it is defined.

\begin{lem} \label{langlemetric}
The function $v_G \colon G \times G \to [0,\pi]$ is
a similarity invariant pseudometric for every domain $G \subsetneq \Rn$. It is a metric unless
$\partial G$ is a proper subset of a line and
will be called the {\it visual angle metric}.
\end{lem}
\begin{proof}
Let $x,\,y\in G$ be arbitrary. Clearly $\ang(x,z,y)=0$ if and only if $x=y$ or $z$ is located on $\ray(x,x-y)$ or $\ray(y,y-x)$. Thus $v_G(x,y)=0$ implies $x=y$, unless $\partial G$ is a proper subset of a line.

We now prove the triangle inequality for $v_G$. Let $x,y,z \in G$. Let
$w$ be an arbitrary point in $E_{xy}^\alpha \cap \bdr G$, where $E_{xy}^\alpha$
is the envelope such that $\alpha$ is the supremum angle in Definition
\ref{maxang}. Without loss of generality, we may assume that $x\,,y\,,z\,,w$ are in $\mathbb{\R}^3$. Let $r=\min\{|w-x|,|w-y|,|w-z|\}/2$, and the
points $x',y'$ and $z'$ denote the intersections of $S^2(w,r)$ with
$[w,x],\; [w,y]$ and $[w,z]$, respectively. Clearly
$$
\ang(x,w,y)=\ang(x',w,y'),\;\ang(z,w,y)=\ang(z',w,y')\,,\ang(z,w,x)=\ang(z',w,x').
$$
Also, by
considering the intrinsic metric of the sphere $S^2(w,r)$ (see \cite[Corollary 18.6.10]{be}), it is clear that
\begin{equation}\label{intsphere}
\ang(x',w,y') \le \ang(x',w,z') + \ang(z',w,y').
\end{equation}
Let $\beta=\ang(x,w,z)$ and $\gamma=\ang(z,w,y)$. By the definition of $v_G$, it is clear that $\beta \le v_G(x,z)$ and
 $\gamma \le v_G(x,z)$. Then by the inequality (\ref{intsphere}),
it now follows that
\begin{eqnarray*}
v_G(x,y)&=&\alpha =\ang(x,w,y) = \ang(x',w,y') \le
 \ang(x',w,z') + \ang(z',w,y')\\
&=& \ang(x,w,z) + \ang(z,w,y)
\le v_G(x,z)+v_G(z,y).
\end{eqnarray*}
This proves the triangle inequality. Similarity invariance is clear, as the shape of
envelopes are similarity invariant.
\end{proof}

 It follows from Remark \ref{fs} and Definition \ref{maxang} that if $G_1\,,G_2$ are proper subdomains of $\Rn$, $G_1 \subset G_2$, and $x\,,y \in G_1$ are distinct points, then $s_{G_1}(x,y) \ge s_{G_2}(x,y)$ and $v_{G_1}(x,y) \ge v_{G_2}(x,y)$. It is also evident that for $x,y \in G$ if there exists $z \in [x,y]$ and $z \notin G$ then $v_G(x,y) = \pi$.
 \end{nonsec}
\begin{nonsec}{\bf M\"obius invariant version of the visual angle metric.}
By Definition \ref{maxang}, Remark \ref{whatisE} and the law of cosines, it
is immediately clear that the metric $v_G$ also has the representation
\begin{eqnarray}\label{omega}
v_G(x,y)&=&\sup_{z \in \bdr G} \arccos \frac{1}{2}\left(\frac{|x-z|}{|y-z|}+\frac{|y-z|}{|x-z|}-\frac{|x-y|^2}{|x-z||y-z|} \right)\\
&=&\sup_{z \in \bdr G} \arccos \frac{1}{2}\left(|z,y,x,\infty|+|z,x,y,\infty|-s(z,x,y,\infty) \right),\nonumber
\end{eqnarray}
where $s(a,b,c,d)=|a,b,d,c||a,c,d,b|$ is the {\it symmetric ratio}. For
properties of the symmetric ratio, see for instance \cite[p.38--39]{vu1}. By
the above representation, it is also easy to immediately verify the similarity
invariance. Next we construct a M\"obius invariant version of the visual angle
metric.

 \begin{lem}\label{varpidef}
The function $\overline{v}_G \colon G \times G \to [0,\pi]$ defined by
\begin{eqnarray}\label{varpi}
\overline{v}_G(x,y)&=&\sup_{z,w \in \bdr G} \arccos \frac{1}{2}\left(|z,y,x,w|+|z,x,y,w|-s(z,x,y,w)\right)\\
&=&\sup_{z,w \in \bdr G} \arccos \frac{1}{2}\bigg(\frac{|x-z||y-w|}{|y-z||x-w|}+\frac{|y-z||x-w|}{|x-z||y-w|}\nonumber\\
&{}&\hspace{25mm}-\frac{|x-y|^2|z-w|^2}{|x-z||x-w||y-z||y-w|} \bigg)\nonumber
\end{eqnarray}
is a M\"obius invariant pseudometric for every domain $G \subsetneq \Rn$. It is a metric
whenever $\partial G$ is not a proper subset of a line or a circle and
will be called the {\it visual double angle metric}.
\end{lem}
\begin{proof} M\"obius invariance is immediate by M\"obius invariance of the
absolute ratio. For the triangle inequality, let $a,b \in \bdr G$, and $f \colon \overline{\Rn} \to
\overline{\Rn}$ be a M\"obius transformation such that $f(b)=\infty$. Then for
$x,y,z \in G$, we get
\begin{eqnarray*}
&{}& \arccos \frac{1}{2}\left(|a,y,x,b|+|a,x,y,b|-s(a,x,y,b)\right)\\
&=& \arccos \frac{1}{2}\left(|f(a),f(y),f(x),\infty|+|f(a),f(x),f(y),\infty|
-s(f(a),f(x),f(y),\infty)\right)\\
&=& \ang(f(x),f(a),f(y)).
\end{eqnarray*}
As in the proof of Lemma \ref{langlemetric}, we see that in the domain $fG$
the inequality
$$\ang(f(x),f(a),f(y)) \le  \ang(f(x),f(a),f(z)) + \ang(f(z),f(a),f(y))$$
holds, so we get
\begin{eqnarray*}
&{}& \arccos \frac{1}{2}\left(|a,y,x,b|+|a,x,y,b|-s(a,x,y,b)\right)\\
&\le&  \ang(f(x),f(a),f(z)) + \ang(f(z),f(a),f(y))\\
&=&\arccos \frac{1}{2}\left(|f(a),f(z),f(x),\infty|+|f(a),f(x),f(z),\infty|
-s(f(a),f(x),f(z),\infty)\right)\\
&{}&+\arccos \frac{1}{2}\left(|f(a),f(z),f(y),\infty|+|f(a),f(y),f(z),\infty|
-s(f(a),f(y),f(z),\infty)\right)\\
&=& \arccos \frac{1}{2}\left(|a,z,x,b|+|a,x,z,b|-s(a,x,z,b)\right)\\
&{}& + \arccos \frac{1}{2}\left(|a,z,y,b|+|a,y,z,b|-s(a,y,z,b)\right)\le
\overline{v}_G(x,z)+\overline{v}_G(y,z),
\end{eqnarray*}
which proves the triangle inequality as the above holds for all points
$b \in \partial G$. The symmetricity and reflexivity are clear. It is also easy
to show that when $f(x)\ne f(y)$, $\ang(f(x),f(a),f(y))$ is zero exactly for point $f(a)$ on
$\ray(f(x),f(x)-f(y))$ or $\ray(f(y),f(y)-f(x))$. Thus the statement follows by the circle preserving property of M\"obius transformations.
\end{proof}

As for the two Ptolemaic angular metrics, the visual angular metric and the visual double angle metric satisfy
an obvious mutual order, which is proved exactly like Proposition
\ref{slessr}, using the formulas (\ref{omega}) and (\ref{varpi}).

\begin{prop}
Let $G \subsetneq \Rn$ and $\infty\in \partial G$, then for all $x,y \in G$
$$v_G(x,y) \le \overline{v}_G(x,y).$$
\end{prop}
\medskip

The following lemma is useful in verifying the triangle inequality.
\begin{lem}\label{fst}{\rm \cite[Exercise 3.33]{vu1}}
Let $f: [0,\infty)\rightarrow [0,\infty)$ be increasing with $f(0)=0$ such that $f(t)/t$ is decreasing on $(0,\infty)$. Then for all $s\,,t\geq 0$
$$f(s+t)\leq f(s)+f(t).$$
\end{lem}

Sometimes it might be more convenient to study the metrics if the
inconvenient inverse cosine function is removed from the definition.

\begin{cor}
The functions $v_G^*$ and $\overline{v}_G^*$ from $G \times G$
onto $[0,1]$, defined for all $x,y \in G\subsetneq \Rn$ by
$$v_G^*(x,y)= \sin \left(\frac{v_G(x,y)}{2} \right)\quad and
\quad \overline{v}_G^*(x,y)= \sin \left(\frac{\overline{v}_G(x,y)}{2} \right).$$
Then $v_G^*$ is a similarity invariant metric provided $\partial G$ is not a proper subset of a line and $\overline{v}_G^*$ is a M\"obius invariant metric provided $\partial G$ is not a proper subset of a line or a circle. Moreover, for all $x,y \in G$, there hold
$$ {v_G(x,y)}/{\pi}\le  v_G^*(x,y) \le
{v_G(x,y)}/{2}\quad and \quad {\overline{v}_G(x,y)}/{\pi} \le  \overline{v}_G^*(x,y) \le{\overline{v}_G(x,y)}/{2}.$$
\end{cor}
\begin{proof}
The function $f \colon x \mapsto \sin(x/2)$ is increasing on $[0,\pi]$,
$f(x)/x$ is decreasing on $(0,\pi)$, and $f(0)=0$. By Lemma \ref{fst}, the triangle inequality follows and hence
$v_G^*$ and $\overline{v}_G^*$ are metrics. The inequalities follow
from the inequality $2x/\pi \le\sin x \le x$ valid in the interval
$x \in [0,\pi/2]$.
\end{proof}

\begin{thm} For all $G \subsetneq \Rn$ and all points $x,y \in G$,
the inequalities
$$v^*_G(x,y) \le s_G(x,y)\qquad and \qquad
\overline{v}_G^*(x,y) \le r_G(x,y)$$
hold.
\end{thm}
\begin{proof}
By (\ref{ef}) we have
\begin{eqnarray*}
 & &E_{xy}^\alpha=\{ z \in \Rn\; :\; \ang(x,z,y) \ge \alpha\}\\
 &\subseteq& F_{xy}^{|x-y|/\sin\frac{\alpha}{2}}
 = \left\{ z \in \Rn\; :\; \frac{|x-y|}{|x-z|+|z-y|} \ge \sin(\alpha/2)\right\}.
\end{eqnarray*}
Thus $v^*_G(x,y) \le s_G(x,y)$, and we see that equality holds
if the boundary point $z\in E_{xy}^\alpha\cap F_{xy}^c$ such that both the angle $\alpha$ and the semimajor axis $c/2$ attain the supremum w.r.t. the domain $G$ in Definition \ref{maxang} and Remark \ref{fs}, respectly.

For the metric $\overline{v}_G^*$ a technique similar to the approach in Lemma
\ref{varpidef} will be used. Let $a,b \in \partial G$, and let $
f \colon \overline{\Rn} \to \overline{\Rn}$ be a M\"obius mapping such
that $f(b)=\infty$. As in the proof of Lemma \ref{varpidef}, we see that
\begin{eqnarray*}
&{}&\arccos \frac{1}{2}\left(|a,y,x,b|+|a,x,y,b|-s(a,x,y,b)\right)\\
&=& \ang(f(x),f(a),f(y))\le\sup_{z \in \partial fG} \ang(f(x),z,f(y)).
\end{eqnarray*}
Since the function $g(x)=\sin(x/2)$ is increasing on $[0,\pi]$, we get
 \begin{eqnarray*}
&{}&\sin\left(\frac{1}{2} \arccos \frac{1}{2}\left(|a,y,x,b|+|a,x,y,b|-s(a,x,y,b)\right)\right)\\
&\le&\sup_{z \in \partial fG} \sin\left(\frac{\ang(f(x),z,f(y))}{2}\right)
=v_{fG}^*(f(x),f(y)).
\end{eqnarray*}
Using the first inequality in this theorem and Proposition \ref{slessr}, we see that
\begin{eqnarray*}
& &\sin\left(\frac{1}{2} \arccos \frac{1}{2}\left(|a,y,x,b|+|a,x,y,b|-s(a,x,y,b)\right)\right)\\
&\le&s_{fG}(f(x),f(y))
\le r_{fG}(f(x),f(y))=r_G(x,y).
\end{eqnarray*}
and the statement follows as $b$ is chosen arbitrarily.
\end{proof}
\end{nonsec}

\bigskip

\section{the visual angle metric in some simple domains}

In this section, we consider the visual angle metric in the punctured space, the unit ball, and the upper half space.

\begin{nonsec} {\bf The punctured space $ G_1=\R^n \setminus \{ 0 \}$.}
For $x,y \in G_1$, we have
$$v_{G_1}(x,y) = \ang(x,0,y)\in [0,\pi]$$
and it is easy to see that $v_{G_1}$ is only a pseudometric.

H. Lind\'en derived the sharp uniformity constant for $G_1$ in \cite[Theorem 1.6]{l2}, i.e., for all $x,y\in G_1$
$$k_{G_1}(x,y)\le \frac{\pi}{\log 3} j_{G_1}(x,y).$$

In $G_1$ the quasihyperbolic metric $k$ and the distance ratio metric $j$ are connected to $v$ as follows
$$v_{G_1}(x,y)=\sqrt{k^2_{G_1}(x,y)-\log^2\frac{|y|}{|x|}}\le k_{G_1}(x,y)$$
and
$$v_{G_1}(x,y)\le \frac{\pi}{\log 3} j_{G_1}(x,y),$$
where the equalities hold if $x=-y$.

\end{nonsec}

\begin{nonsec} {\bf The unit ball $G_2 = \Bn$.}
For the convenience of geometric explanation, let $x,y \in \BB$ and $x\neq y$. We define ellipses
$$E_x = \{ z \in \BB \colon |x-z|+|z|=1 \}, \quad E_y = \{ z \in \BB \colon |y-z|+|z|=1 \}$$
and denote $E_x \cap E_y = \{ z_1,z_2 \}$ (See Figure \ref{angmetfig34}(a)). We choose $z$ to be that one of the points $z_1$ and $z_2$, which has the larger norm. Then
$$v_{\BB}(x,y) = \frac 12\ang (x,z,y).$$

In particular, for $x\neq0\,,y=0$, we have
\begin{eqnarray}\label{omega0x}
v_{\BB}(0,x) = \arcsin |x|\in (0,\pi/2),
\end{eqnarray}
and for $|x|=|y|\neq 0$, $\theta=\frac 12 \ang(x,0,y)\in (0, \pi/2]$, we have
\begin{eqnarray}\label{omega1x}
v_{\BB}(x,y) = 2\arctan\frac{|x|\sin\theta}{1-|x|\cos\theta}.
\end{eqnarray}

\begin{figure}[h]
\subfigure[]{\includegraphics[width=.28 \textwidth]{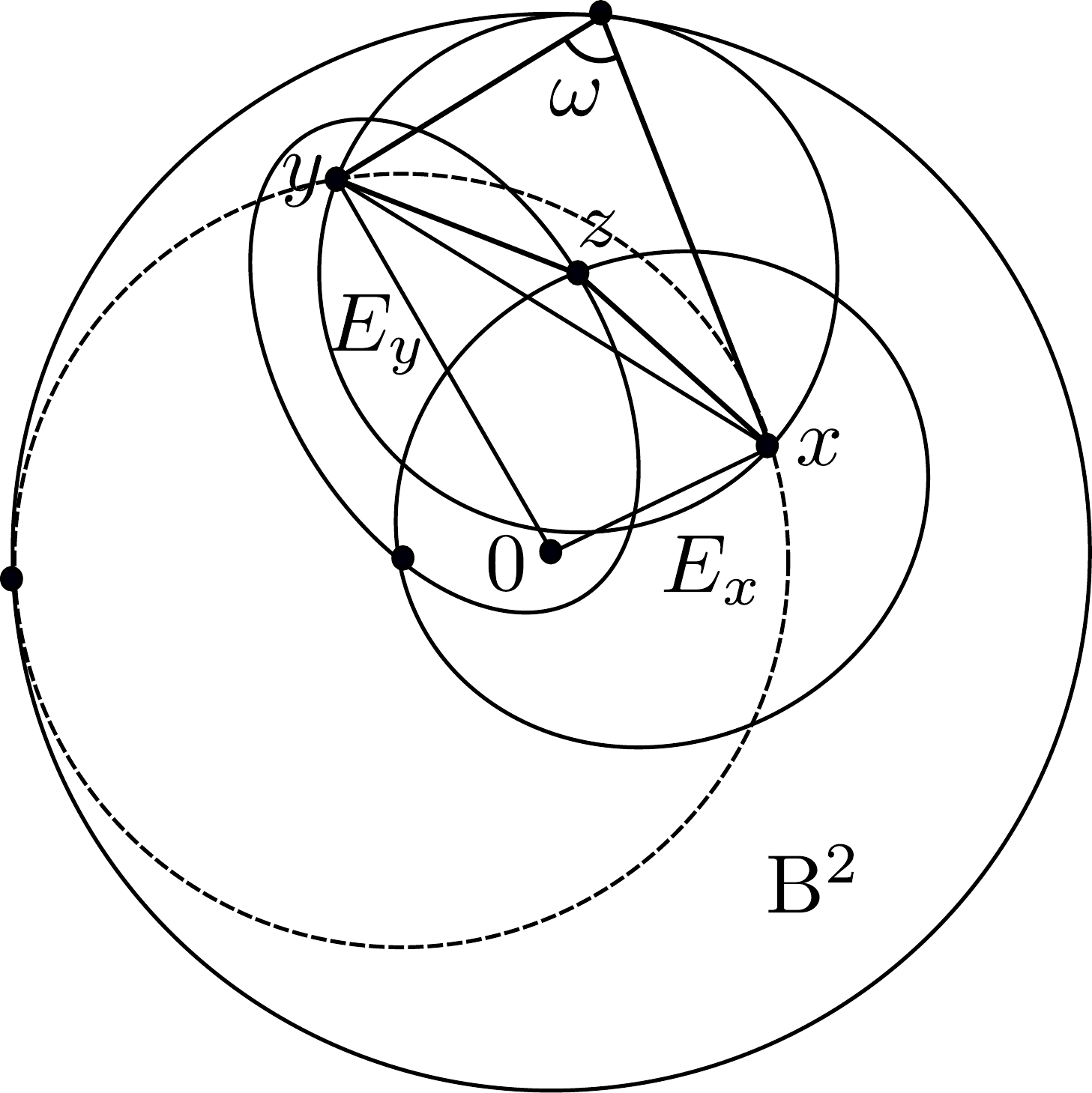}}
\hspace{.04 \textwidth}
\subfigure[]{\includegraphics[width=.28 \textwidth]{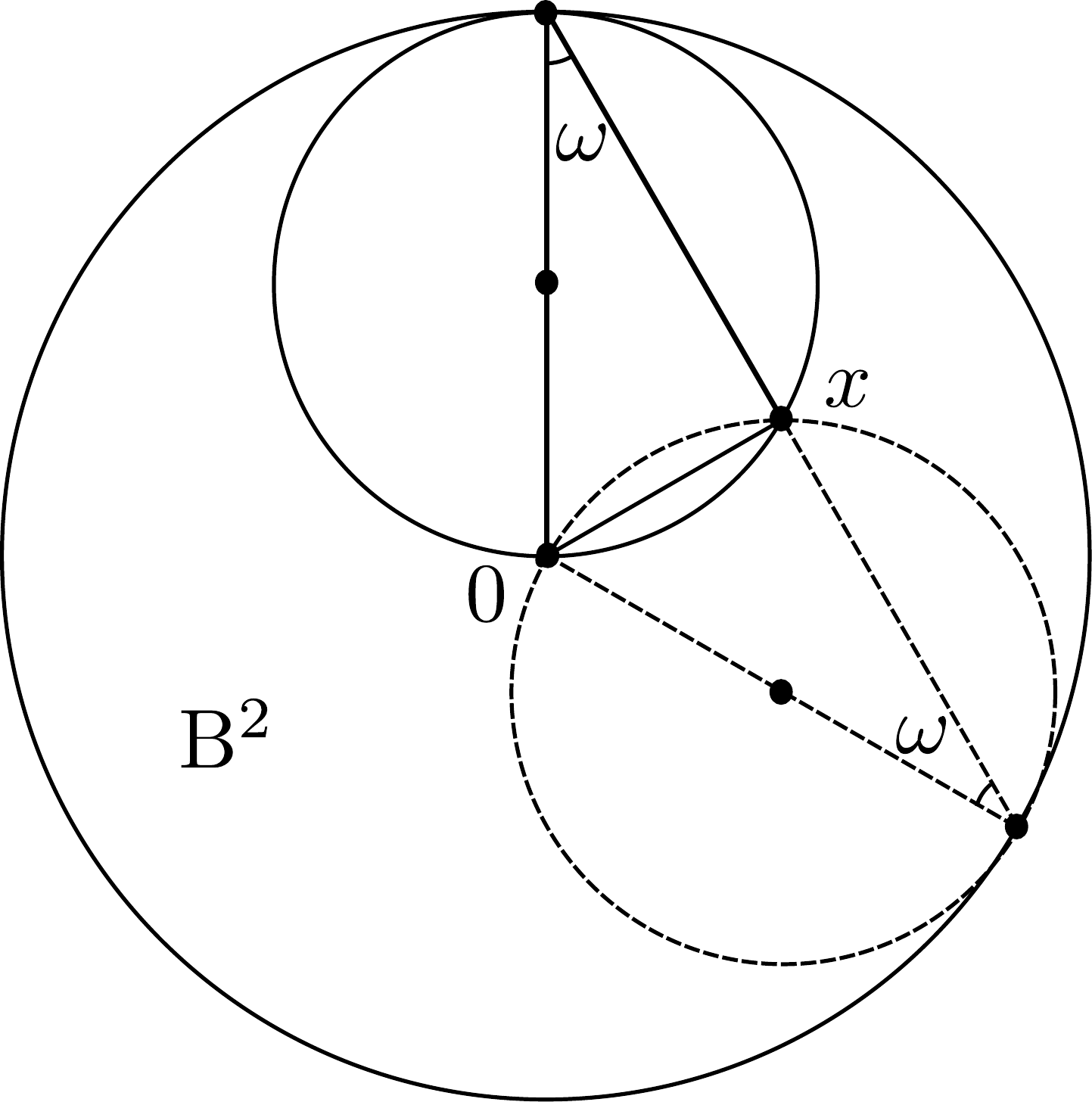}}
\hspace{.04 \textwidth}
\subfigure[]{\includegraphics[width=.28 \textwidth]{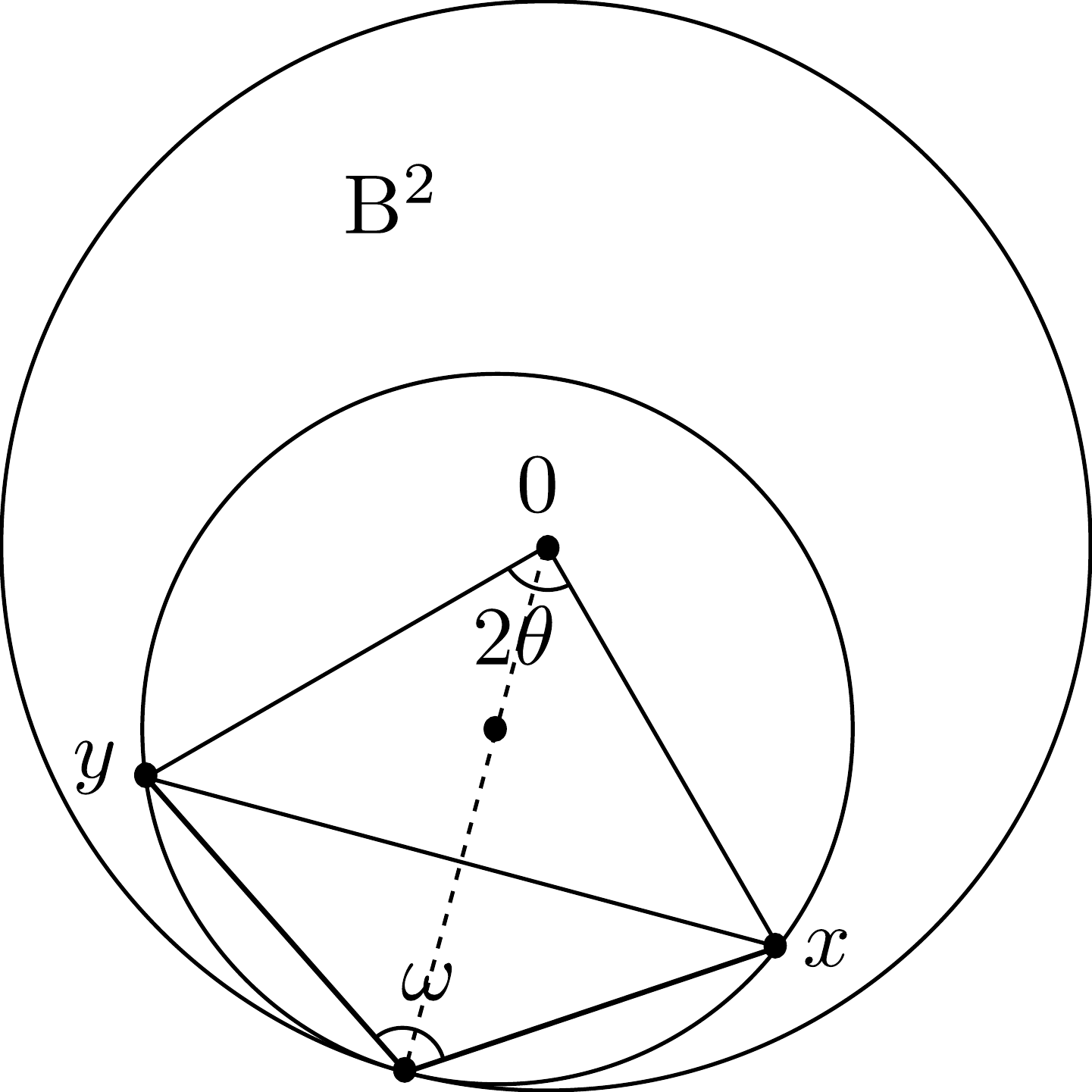}}
\caption{\label{angmetfig34} The visual angle metric in the unit disk $v_{\BB}(x,y)=\omega$. (a) General case, where $z$ is in the intersection of ellipses $E_x$ and $E_y$. (b) Special case \eqref{omega0x}, where $y=0$. (c) Special case \eqref{omega1x}, where $|x|=|y|$ and $\ang(x,0,y)=2\theta$.}
\end{figure}

For the comparison of the visual angle metric and the hyperbolic metric in the unit ball, we need some technical lemmas.

The next lemma, so-called {\em
monotone form of l'H${\rm \hat{o}}$pital's rule}, has found recently numerous applications in proving inequalities. See the extensive bibliography of \cite{avz}.

\begin{lem}\label{lhr}{\rm \cite[Theorem 1.25]{avv1}}
Let $-\infty<a<b<\infty$, and let $f,\,g: [a,b]\rightarrow \mathbb{R}$ be continuous on $[a,b]$, differentiable on $(a,b)$. Let $g'(x)\neq 0$ on $(a,b)$.Then if $f'(x)/g'(x)$ is increasing(decreasing) on $(a,b)$, so are
\begin{eqnarray*}
\frac{f(x)-f(a)}{g(x)-g(a)}\,\,\,\,\,\,\,and\,\,\,\,\,\,\,\,\frac{f(x)-f(b)}{g(x)-g(b)}.
\end{eqnarray*}
If $f'(x)/g'(x)$ is strictly monotone, then the monotonicity in the conclusion is also strict.
\end{lem}

\begin{lem}\label{le1}
(1) The function $f_1(r)\equiv \frac{\arcsin r}{{\rm arth}\, r}$ is strictly decreasing from $(0,1)$ onto $(0,1)$.\\
(2) The function $f_2(r)\equiv \frac{\arcsin r}{\log( 1/(1-r))}$ is strictly decreasing from $(0,1)$ onto $(0,1)$.\\
(3) The function $f_3(r)\equiv\arctan\frac{c r}{1-c\sqrt{1-r^2}}-{\rm arsh}\,\frac{2c r}{1-c^2}$ is strictly decreasing from $(0,1)$ onto $(\arctan\,c-\log\frac{1+c}{1-c}\,, 0)$ for $c\in(0,1)$.\\
(4) The function $f_4(r)\equiv\frac{\arctan r}{{\rm arch}\,(1+2r^2)}$ is strictly decreasing from $(0,\infty)$ onto $(0,1/2)$.
\end{lem}
\begin{proof}
(1) Let $f_{11}(r)=\arcsin r$ and $f_{12}(r)={\rm arth}\, r$. Then $f_{11}(0^+)=f_{12}(0^+)=0$. By differentiation,
$$\frac{f'_{11}(r)}{f'_{12}(r)}=\sqrt{1-r^2},$$
which is strictly decreasing on $(0,1)$. Therefore, $f_1$ is strictly decreasing on $(0,1)$ by Lemma \ref{lhr}. The limiting value $f_1(1^-)=0$ is clear and $f_1(0^+)=1$ by l'H\^opital's Rule.

(2) Let $f_{21}(r)=\arcsin r$ and $f_{22}(r)=\log( 1/(1-r))$. Then $f_{21}(0^+)=f_{22}(0^+)=0$. By differentiation,
$$\frac{f'_{21}(r)}{f'_{22}(r)}=\sqrt{\frac{1-r}{1+r}},$$
which is strictly decreasing on $(0,1)$. Therefore, $f_2$ is strictly decreasing on $(0,1)$ by Lemma \ref{lhr}. The limiting value $f_2(1^-)=0$ is clear and $f_2(0^+)=1$ by l'H\^opital's Rule.

(3) Let $r'=\sqrt{1-r^2}$. By differentiation,
$$f'_3(r)=\frac{c}{\sqrt{1+c^2-2c r'}}\left(\frac{r'-c}{r'\sqrt{1+c^2-2c r'}}-\frac{2}{\sqrt{1+c^2+2c r'}}\right).$$
It is clear that $f'_3(r)<0$ if $r'\le c$. Therefore, we suppose that $r'>c$, namely, $0<r<\sqrt{1-c^2}$ in the sequel.
Rewrite
$$f'_3(r)=\frac{2c}{\sqrt{(1+c^2)^2-(2c r')^2}}\left(\frac 12 \phi(r)-1\right),$$
where $\phi(r)=\frac{r'-c}{r'}\sqrt{\frac{1+c^2+2cr'}{1+c^2-2cr'}}$
is strictly decreasing. Therefore, we have $\phi(r)<\phi(0)=1+c$ and hence $f'(r)<0$ when $r'>c$.

Therefore, $f_3$ is strictly decreasing on $(0,1)$. The limiting values are clear.

(4) Let $f_{41}(r)=\arctan r$ and $f_{42}(r)={\rm arch}\,(1+2r^2)$. Then $f_{41}(0^+)=f_{42}(0^+)=0$. By differentiation,
$$\frac{f'_{41}(r)}{f'_{42}(r)}=\frac{1}{2\sqrt{1+r^2}},$$
which is strictly decreasing on $(0,\infty)$. Therefore, $f_4$ is strictly decreasing by Lemma \ref{lhr}. The limiting value $f_4(\infty)=0$ is clear and $f_4(0^+)=1/2$ by l'H\^opital's Rule.
\end{proof}

\begin{lem}\label{at}
 Let $\alpha\in(0,\pi)$. Then the function
 $$f(\theta)=(1+\cos(\alpha+\theta))(1+\cos(\alpha-\theta))$$
 is strictly decreasing from $(0,\pi-\alpha)$ onto $(0, (1+\cos\alpha)^2)$.
\end{lem}
\begin{proof}
Since $0<\theta<\pi-\alpha<\pi$, we have $\cos\theta>-\cos\alpha$.
Therefore,
$$f'(\theta)=-2\sin\theta(\cos\theta+\cos\alpha)<0.$$
Hence $f(\theta)$ is strictly decreasing on $(0,\pi-\alpha)$. The limiting values are clear.
\end{proof}

\begin{lem}\label{le4}
 Let $a\in\Bn$, $P$ be any hyperplane through $0$ and $a$. Let $C=S(a,r)$ be a circle centered at $a$ with radius $r$ in $\overline{\Bn}\cap P$ and tangent to $S^{n-1}$ at the point $z$.
 Let two distinct points $x'\,,y'\in C$ such that $|x'|=|y'|$ and $\ang(x',z,y')=\alpha\in(0,\pi)$. Then
for arbitrary two distinct points $x\,,y\in C$ with $\ang(x,z,y)=\alpha$, there holds
$$(1-|x|^2)(1-|y|^2)\leq(1-|x'|^2)(1-|y'|^2).$$
\end{lem}

\begin{proof}

Without loss of generality, we may assume that $\overline{\Bn}\cap P=
\overline{\BB}$. Choose two distinct points $x\,,y\in C$ and $\ang(x,z,y)=\alpha$. By symmetry, we may also assume that $|x|\leq|y|$, and the triples $(x,z,y)$ and $(x',z,y')$ are labeled in the positive order on $C$, respectively (see Fig \ref{angmetfig5}).

It is easy to see that $r=1-|a|$ and $[x',y']\perp L(0,z)$, namely, $x'\,,y'$ are symmetry with respect to $L(0,z)$. For the inequality, we divide the proof into three cases.

\begin{figure}[ht]
\subfigure[]{\includegraphics[width=.28 \textwidth]{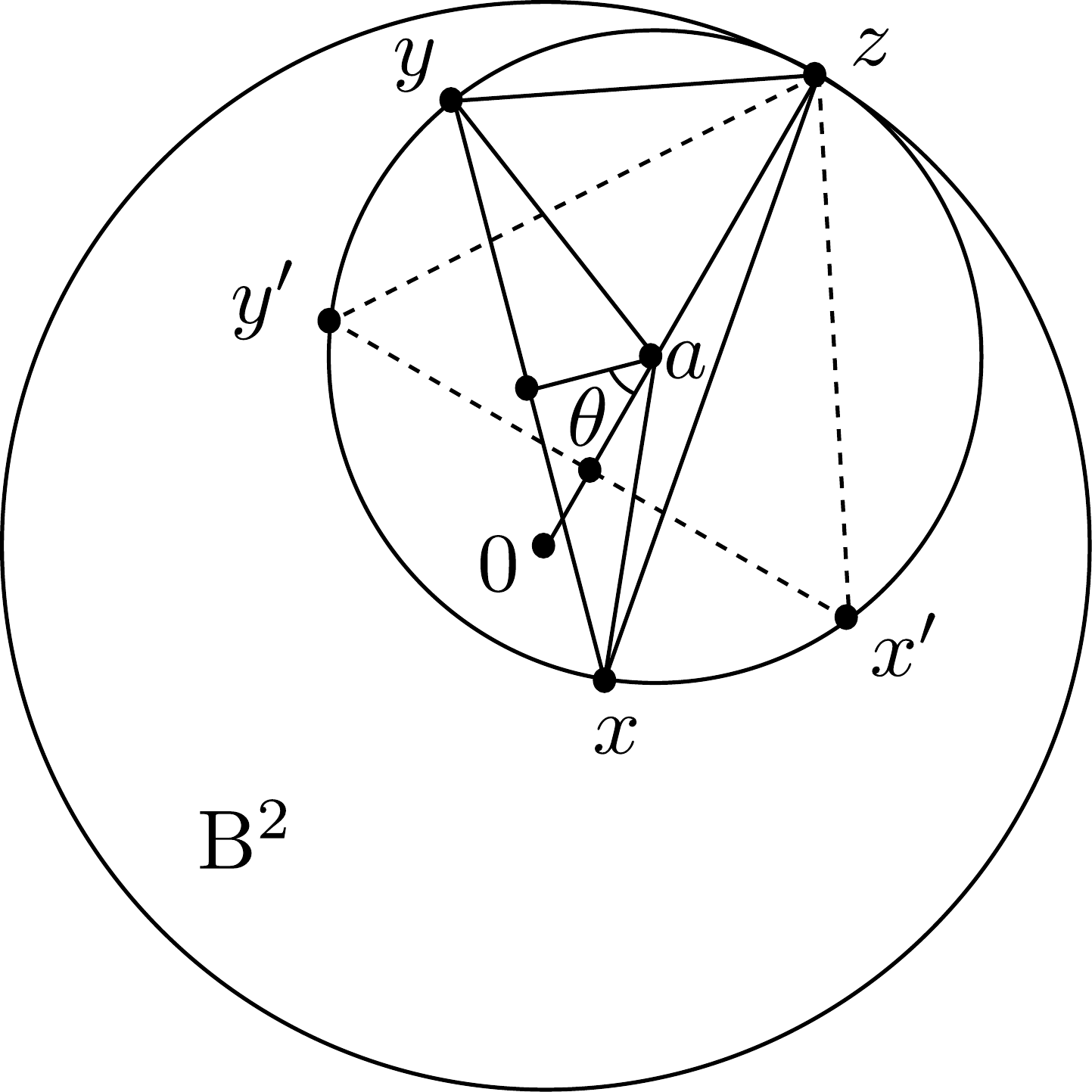}}
\hspace{.04 \textwidth}
\subfigure[]{\includegraphics[width=.28 \textwidth]{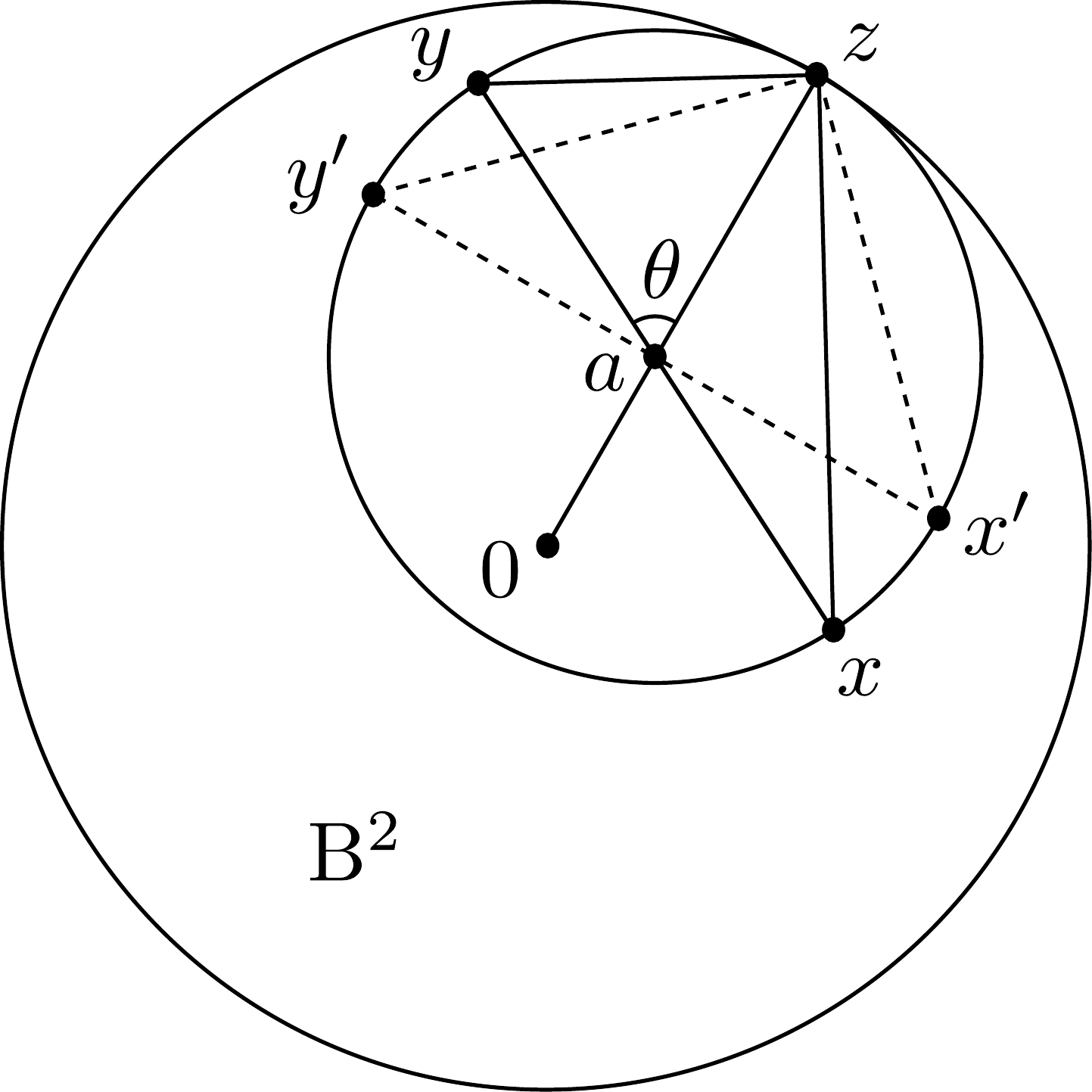}}
\hspace{.04 \textwidth}
\subfigure[]{\includegraphics[width=.28 \textwidth]{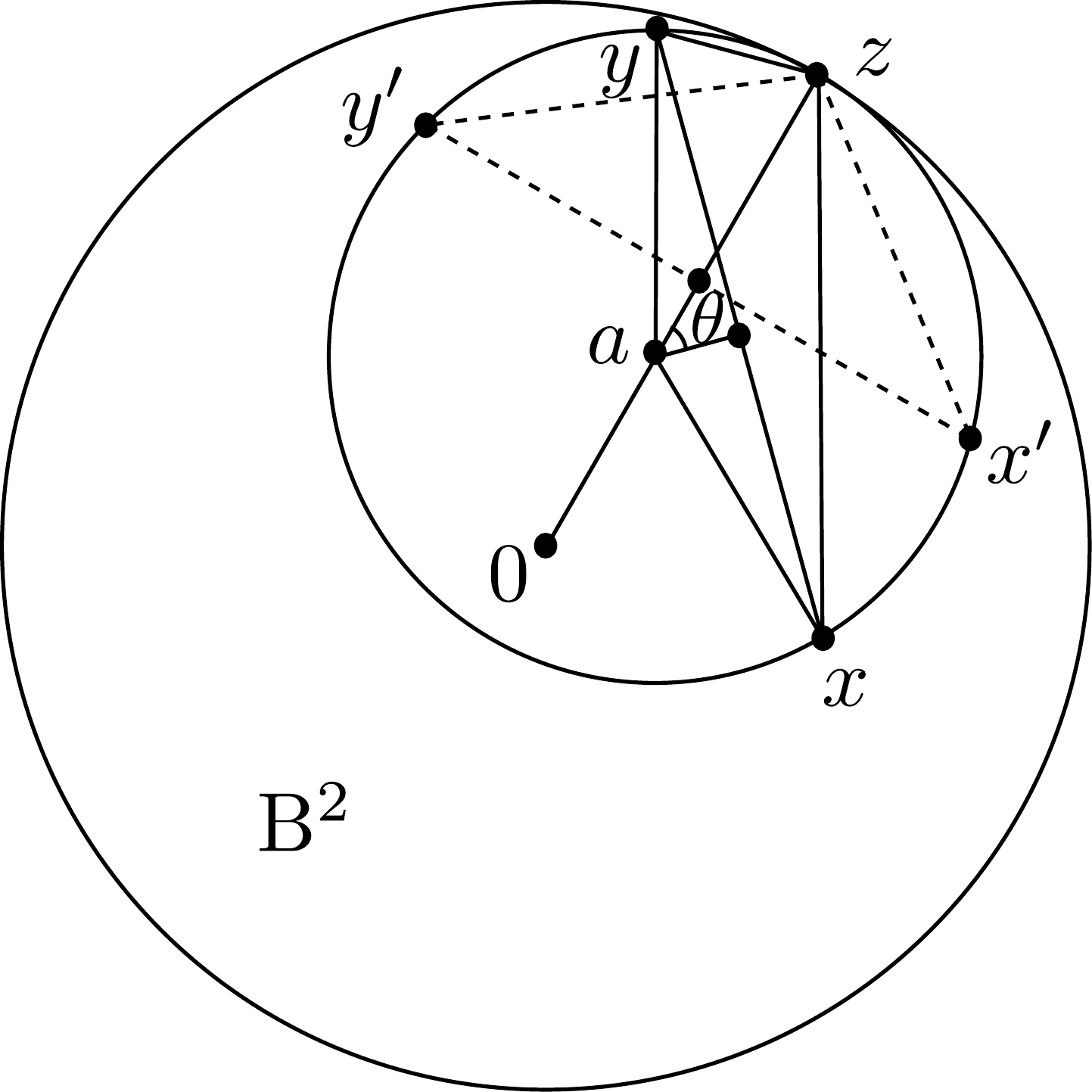}}
\caption{\label{angmetfig5} Proof of Lemma \ref{le4}. Here $\ang(x,z,y)=\ang(x',z,y')=\alpha$ and $(1-|x|^2)(1-|y|^2)\leq(1-|x'|^2)(1-|y'|^2).$}
\end{figure}

{\it Case 1.} $0<\alpha<\frac{\pi}{2}$.
Let $\theta=\ang(0,a,\frac{x+y}{2})\in[0,\pi-\alpha)$. It is clear that $\ang(0,a,\frac{x'+y'}{2})=0.$
Then $$x=a-\frac{a}{|a|}(1-|a|)e^{i(\alpha-\theta)}\,\,\,\,{\rm and}\,\,\,\,y=a-\frac{a}{|a|}(1-|a|)e^{-i(\alpha+\theta)}.$$
Hence by Lemma \ref{at}, we have
$$(1-|x|^2)(1-|y|^2)=4 |a|^2(1-|a|)^2  f(\theta)\leq4 |a|^2(1-|a|)^2 f(0),$$
where $f(\theta)=(1+\cos(\alpha+\theta))(1+\cos(\alpha-\theta))$. Namely,
$$(1-|x|^2)(1-|y|^2)\leq(1-|x'|^2)(1-|y'|^2).$$

{\it Case 2.} $\alpha=\frac{\pi}{2}$.
Let $\theta=\ang(z,a,y)\in(0,\pi/2]$. It is clear that $\ang(z,a,x')=\ang(z,a,y')=\frac{\pi}{2}$.
Then $$x=a+\frac{a}{|a|}(1-|a|)e^{-i(\pi-\theta)}\,\,\,\,{\rm and}\,\,\,\,y=a+\frac{a}{|a|}(1-|a|)e^{i \theta}.$$
Hence  we have
$$(1-|x|^2)(1-|y|^2)= 4|a|^2(1-|a|)^2\sin^2\theta\leq4|a|^2(1-|a|)^2 \sin^2\frac{\pi}{2}.$$
Therefore, we have
$$(1-|x|^2)(1-|y|^2)\leq(1-|x'|^2)(1-|y'|^2).$$

{\it Case 3.} $\frac{\pi}{2}<\alpha<\pi$.
Let $\theta=\ang(z,a,\frac{x+y}{2})\in[0,\pi-\alpha)$. It is clear that $\ang(z,a,\frac{x'+y'}{2})=0.$
Then $$x=a+\frac{a}{|a|}(1-|a|)e^{-i(\pi-\alpha+\theta)}\,\,\,\,{\rm and}\,\,\,\,y=a+\frac{a}{|a|}(1-|a|)e^{i(\pi-\alpha-\theta)}.$$
 By a similar argument to Case 1,  we have
$$(1-|x|^2)(1-|y|^2)=4|a|^2(1-|a|)^2 f(\theta)\leq4|a|^2(1-|a|)^2 f(0),$$
where $f(\theta)=(1+\cos(\alpha+\theta))(1+\cos(\alpha-\theta))$. Namely,
$$(1-|x|^2)(1-|y|^2)\leq(1-|x'|^2)(1-|y'|^2).$$

By Case 1-3, we complete the proof.
\end{proof}

\begin{cor}\label{cor1}
Let $x\,,y\,,x'\,,y'$ be as in Lemma \ref{le4}. Then
$$\rho_{\Bn}(x,y)\geq\rho_{\Bn}(x',y').$$
\end{cor}

\begin{lem}\label{le2}  Let $x\in \Bn$. Then
$$v_{\Bn}(0,x) \le\frac12\rho_{\Bn}(0,x),$$
and
$$v_{\Bn}(0,x) \le j_{\Bn}(0,x).$$
The constant $\frac 12$ in the first inequality is the best possible and the second inequality is sharp.
\end{lem}
\begin{proof}
For all $x\in \Bn$ and $x\neq 0$, by (\ref{arth}) we have
$$\rho_{\Bn}(0,x)=2{\rm arth}\,|x|\,\,\,\,{\rm and} \,\,\,\, j_{\Bn}(0,x)=\log\frac{1}{1-|x|}.$$
By (\ref{omega0x}), Lemma \ref{le1}(1)--(2), we obtain the inequalities, which are sharp if $|x|\rightarrow 0^+$.
\end{proof}

\begin{lem} \label{le3}
Let $x\,,y\in \Bn$ and $|x|=|y|$. Then
$$v_{\Bn}(x,y) \le \rho_{\Bn}(x,y),$$
and the inequality is sharp.
\end{lem}
\begin{proof}
Let $|x|=|y|\in(0,1)$ and $\theta=\frac 12\ang(x,0,y)\in (0, \pi/2]$. Then by (\ref{sinh})
$$\rho_{\Bn}(x,y)=2 {\rm arsh}\frac{2|x|\sin\theta}{1-|x|^2}.$$
By (\ref{omega1x}) and making substitution of $r=\sin\theta$ and $c=|x|$ in Lemma \ref{le1}(3), we have
$$v_{\Bn}(x,y) \le \rho_{\Bn}(x,y).$$
For the sharpness, let $|x|=|y|=1-1/t$ and $\sin\theta=e^{-t}$ ($t>0$). Then by l'H\^opital's Rule
$$\lim_{t\rightarrow+\infty}\frac{v_{\Bn}(x,y)}{\rho_{\Bn}(x,y)}=\lim_{t\rightarrow+\infty}\frac{\arctan\frac{1-1/t}{e^t[1-(1-1/t)\sqrt{1-e^{-2t}}]}}{{\rm arsh}\frac{2(1- 1/t)}{e^t[1-(1-1/t)^2]}}=\lim_{t\rightarrow+\infty}\frac{1-(1-1/t)^2}{2[1-(1-1/t)\sqrt{1-e^{-2t}}]}=1.$$
Together with Lemma \ref{le2}, we obtain the result.
\end{proof}

\begin{thm}\label{vrho1}
Let $x\,,y\in \Bn$. Then
$$v_{\Bn}(x,y) \le \rho_{\Bn}(x,y),$$
and the inequality is sharp.
\end{thm}
\begin{proof}
For $x\,,y\in \Bn$, $x\neq y$ and $|x|\neq|y|$, by Lemma \ref{le4}, Lemma \ref{le3}, Corollary \ref{cor1} there exist $x'\,,y'\in \Bn$ such that $|x'|=|y'|$, $|x-y|=|x'-y'|$ and
$$v_{\Bn}(x,y)=v_{\Bn}(x',y')\le \rho_{\Bn}(x',y')\leq \rho_{\Bn}(x,y).$$
Together with Lemma \ref{le3}, the inequality holds for all  $x\,,y\in \Bn$ and it is sharp.
\end{proof}

\begin{conjecture}
There exists a constant $c\in(1.431, 1.432)$ such that for all $x\,,y\in \Bn$
$$v_{\Bn}(x,y) \le c j_{\Bn}(x,y).$$
 \end{conjecture}

The following lemma shows the equality case between the visual angle metric and the hyperbolic metric in the unit ball.
\begin{lem}\label{mthmble}
Let $x\,,y\in \Bn$. Then
$$\tan v_{\Bn}(x,y)={\rm sh}\frac{\rho_{\Bn}(x,y)}{2}$$
if and only if $0\,,x\,,y$ are collinear or one of the two points $x\,,y$ is $0$.
\end{lem}
\begin{proof}
It suffices to consider the 2-dimensional case.

For $x\,,y\in \BB$ and $x\neq y$. Let $z\in E^{\omega}_{xy}\cap S^1$, where $E^{\omega}_{xy}$ is the envelope such that $\omega$ is the supremum angle in Definition \ref{maxang}, i.e., $v_{\Bn}(x,y)=\ang(x,z,y)=\omega$.
Then there exists exactly one circle $S^1(a, 1-|a|)$ which passes through $x\,,y\,,z$ and is tangent to $S^1$. For the convenience of proof, we suppose that the three points $x\,,z\,,y$ are labeled in the positive order on $S^1(a, 1-|a|)$. By symmetry, we may assume that $|x|\le |y|$. By geometric observation, $v_{\BB}(x,y)\in(0,\pi/2)$ if $0\,,x\,,y$ are collinear or one of the two points $x\,,y$ is $0$ (cf. Figure \ref{angmetfig5}(a)) or $\tan v_{\BB}(x,y)={\rm sh}\frac{\rho_{\BB}(x,y)}{2}$. It is clear that
$$|x-y|=2 (1-|a|) \sin\omega.$$
Let $\theta=\ang(0,a,\frac{x+y}{2})$, then by the proof of Case 1 in Lemma \ref{le4}, we get
$$(1-|x|^2)(1-|y|^2)=4 |a|^2(1-|a|)^2f(\theta),$$
where $f(\theta)$ is as in Lemma \ref{at} by taking $\alpha=\omega$.
Hence, we have
$$\tan v_{\BB}(x,y)={\rm sh}\frac{\rho_{\BB}(x,y)}{2}\Leftrightarrow |a|\sqrt{f(\theta)}=\cos\omega.$$

Let $s=|a|$ and $|\frac{x+y}2|=t$ in the sequel.
If one of the two points $x\,,y$ is $0$, then $0\in S^1(a, 1/2)$ and $\theta=\omega$. Hence
$$\tan v_{\BB}(0,y)={\rm sh}\frac{\rho_{\BB}(0,y)}{2}\Leftrightarrow s\sqrt{f(\omega)}=\cos\omega,$$
and the last equality clearly holds.

If $0\,,x\,,y$ are collinear, we consider the following two cases.

\medskip
{\it Case 1.} $0\in \BB(a, 1-s)$.
By the law of cosines
$$\cos(\omega+\theta)=\frac{s^2+(1-s)^2-((1-s)\sin\omega+t)^2}{2 s(1-s)}$$
and
$$\cos(\omega-\theta)=\frac{s^2+(1-s)^2-((1-s)\sin\omega-t)^2}{2 s(1-s)}.$$
Since
$$((1-s)\sin\omega+t)((1-s)\sin\omega-t)=1-2s,$$
we get
$$t^2=(1-s)^2\sin^2\omega-(1-2s).$$
Thus we have
\begin{eqnarray*}
4s^2(1-s)^2\,f(\theta)&=&(1-(1-s)^2\sin^2\omega-t^2)^2-4t^2(1-s)^2\sin^2\omega\\
&=&4 (1-s)^2\cos^2\omega.
\end{eqnarray*}
Therefore,
$$s\sqrt{f(\theta)}=\cos\omega.$$

{\it Case 2.} $0\in \BB\setminus \overline{\BB(a, 1-s)}$.
By the law of cosines
$$\cos(\omega+\theta)=\frac{s^2+(1-s)^2-(t+(1-s)\sin\omega)^2}{2 s(1-s)}$$
and
$$\cos(\omega-\theta)=\frac{s^2+(1-s)^2-(t-(1-s)\sin\omega)^2}{2 s(1-s)}.$$
Since
$$(t+(1-s)\sin\omega)(t-(1-s)\sin\omega)=2s-1,$$
we have, by a similar argument as the proof of Case 1,
$$s \sqrt{f(\theta)}=\cos\omega.$$

By Case 1-2, we conclude that $0\,,x\,,y$ are collinear implies $\tan v_{\BB}(x,y)={\rm sh}\frac{\rho_{\BB}(x,y)}{2}$.

Next, suppose that $0\,,x\,,y$ are noncollinear and $\tan v_{\BB}(x,y)={\rm sh}\frac{\rho_{\BB}(x,y)}{2}$. Then there exists two points $x'\,,y'\in S^1(a,1-s)$ such that $0\,,x'\,,y'$ are collinear and $|x'-y'|=|x-y|$. Then by the above proof and the monotonicity of $f$, we have
$$\tan v_{\BB}(x,y)=\tan v_{\BB}(x',y')={\rm sh}\frac{\rho_{\BB}(x',y')}{2}\neq{\rm sh}\frac{\rho_{\BB}(x,y)}{2},$$
which is a contradiction. Therefore, if neither of the two points $x\,,y$ is $0$ and $\tan v_{\BB}(x,y)={\rm sh}\frac{\rho_{\BB}(x,y)}{2}$, then $0\,,x\,,y$ are collinear.

This completes the proof.
\end{proof}

\begin{thm}\label{mthmb}
Let $x\,,y\in \Bn$. Let $\rho_{\Bn}^*(x,y)=\arctan\left({\rm sh}\frac{\rho_{\Bn}(x,y)}{2}\right)$. Then
$$\rho_{\Bn}^*(x,y)\leq v_{\Bn}(x,y)\leq 2\rho^*_{\Bn}(x,y).$$
The equality holds in the left-hand side if and only if $0\,,x\,,y$ are collinear or one of the two points $x\,,y$ is $0$,
and the constant 2 in the right-hand side of the inequality is the best possible.
\end{thm}
\begin{proof}
It suffices to consider the 2-dimensional case.

For $x\,,y\in \BB$ and $x\neq y$. In the same way as in the proof of Lemma \ref{mthmble}, we have the point $z\in S^1$ and the angle $\omega$ such that $v_{\Bn}(x,y)=\ang(x,z,y)=\omega$. We also have the circle $S^1(a, 1-|a|)$ which passes through $x\,,y\,,z$ and is tangent to $S^1$.
By Lemma \ref{le4}, there exist $x'\,,y'\in S^1(a, 1-|a|)$ such that $\ang(x',z,y')=\ang(x,z,y)$ and $|x'|=|y'|$.
For the convenience of proof, we suppose that the three points $x\,,z\,,y$ are labeled in the positive order on $S^1(a, 1-|a|)$, and so are the the points $x'\,,z\,,y'$.
Without loss of generality, we may still assume that $|x|\le |y|$ (cf. Figure \ref{angmetfig5}).
By the proof of Lemma \ref{le4}, we have
$$(1-|x'|^2)(1-|y'|^2)=4 |a|^2(1-|a|)^2(1+\cos\omega)^2$$
and
$$|x'-y'|=|x-y|=2 (1-|a|) \sin\omega.$$

Therefore, by Lemma \ref{le4}
\begin{eqnarray}\label{ub}
\tan\frac{\omega}{2}&=&|a|\frac{|x'-y'|}{\sqrt{(1-|x'|^2)(1-|y'|^2)}}\nonumber\\
&\leq &\frac{|x-y|}{\sqrt{(1-|x|^2)(1-|y|^2)}}={\rm sh}\frac{\rho_{\BB}(x,y)}{2}.
\end{eqnarray}
Thus we prove the right-hand side of the inequality.

For the sharpness, let $x=(1-t)+i\, t$ and $y=(1-t)-i\, t\,(0<t<1)$. Then $x\,,y\in S^1(1-t,t)$ and $|x-y|=2t$. Therefore, we have
$$\lim_{t\rightarrow 0^+}\frac{v_{\BB}(x,y)}{\rho^*_{\BB}(x,y)}=\lim_{t\rightarrow 0^+}\frac{\pi}2\left(\arctan \frac1{1-t}\right)^{-1}=2.$$

To prove the left-hand side of the inequality, we only need to consider $v_{\Bn}(x,y)\in (0,\pi/2)$ because ${\rm sh}\frac{\rho_{\Bn}(x,y)}{2}$ is always nonnegative. The equality clearly holds if one of the two points $x\,,y$ is $0$ by Lemma \ref{mthmble}. We consider the case $x\neq 0$ in the sequel. Let $x''\,,y''$ be two points such that the three points $x''\,,z\,,y''$ are labeled in the positive order on $S^1(a, 1-|a|)$, and $|x''-y''|=|x-y|$, and $0\,,x''\,,y''$ are collinear. Then by the definition of the visual angle metric, Lemma \ref{mthmble} and Lemma \ref{at}, we get
$$v_{\BB}(x,y)=v_{\BB}(x'',y'')=\rho^*_{\BB}(x'',y'')\ge \rho^*_{\BB}(x,y).$$



Therefore,the left-hand side of the inequality is proved, and the equality is clear by Lemma \ref{mthmble}.

This completes the proof.
\end{proof}

\begin{cor}\label{vrho2}
For all $x\,,y\in \Bn$, we have
$$v_{\Bn}(x,y) \le 2\arctan\frac{|x-y|(2-|x-y|)}{2\sqrt{(1-|x|^2)(1-|y|^2)}},$$
and the equality holds if $|x|=|y|=(\sqrt 2\sin(\theta+\frac{\pi}{4}))^{-1}$ and $\theta=\frac12\ang(x,0,y)\in(0, \pi/2)$.
\end{cor}
\begin{proof}
We still consider the 2-dimensional case. Let $a$ be as in the proof of Theorem \ref{mthmb}.
 Since $|a|+|a-x|=|a|+|a-y|=1$, we get
 $$|a|=\frac{2-(|a-x|+|a-y|)}{2}\le \frac{2-|x-y|}{2}.$$
Then by (\ref{ub}), we prove the inequality.

For the equality, let $|x|=|y|>0$ and $\theta=\frac 12\ang(x,0,y)>0$, then
\be\label{v1}
\frac{|x-y|(2-|x-y|)}{2\sqrt{(1-|x|^2)(1-|y|^2)}}=\frac{2|x|\sin\theta(1-|x|\sin\theta)}{1-|x|^2}.
\ee
By (\ref{omega1x}) and (\ref{v1}) the equality holds if $|x|=|y|=\frac{1}{\sin\theta+\cos\theta}$.
\end{proof}

\end{nonsec}

\medskip

\begin{nonsec}{\bf The upper half space $G_3 = \Hn$.}
 For $G_3$ it is sufficient to consider the case $n=2$. Let $x=(x_1,x_2),y = (y_1,y_2) \in \UH$ and $x\neq y$. Then the circle through $x\,,y$ and tangent to $\partial\UH$ with center
 $$z=\frac{x_1y_2-x_2y_1+\sqrt{x_2y_2}|x-y|}{y_2-x_2}+i\frac{(x_2+y_2)|x-y|^2+2\sqrt{x_2y_2}(x_1-y_1)|x-y|}{2(y_2-x_2)^2}$$
 or
  $$z'=\frac{x_1y_2-x_2y_1-\sqrt{x_2y_2}|x-y|}{y_2-x_2}+i\frac{(x_2+y_2)|x-y|^2-2\sqrt{x_2y_2}(x_1-y_1)|x-y|}{2(y_2-x_2)^2}$$
if $x_2\neq y_2$,
and
$$w=\frac{x_1+y_1}{2}+i\frac{4x^2_2+(x_1-y_1)^2}{8x_2}$$
if $x_2=y_2$ (see Figure \ref{angmetfig67}).

\begin{figure}[h]
\subfigure[]{\includegraphics[width=.46 \textwidth]{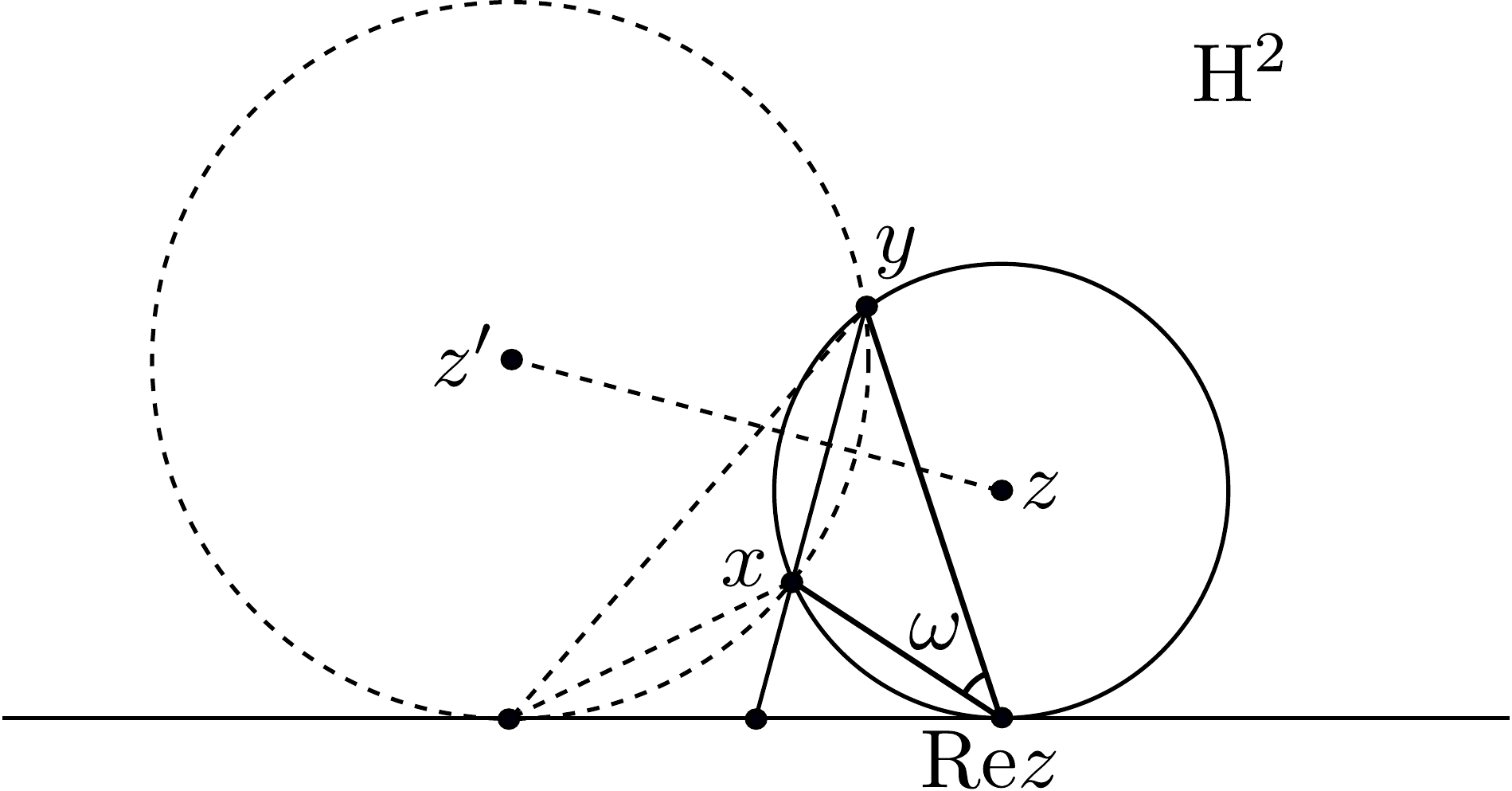}}
\hspace{.03 \textwidth}
\subfigure[]{\includegraphics[width=.38 \textwidth]{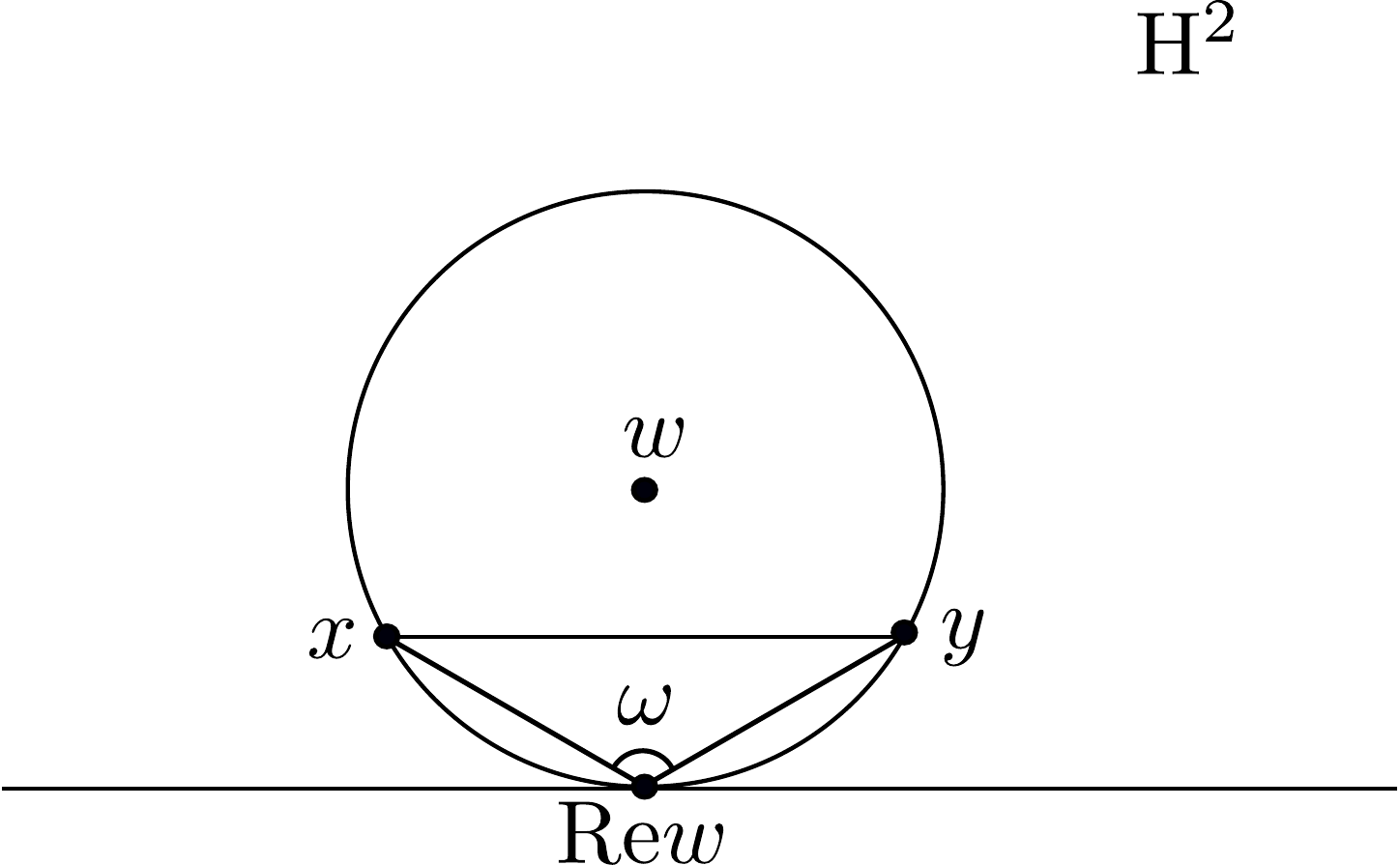}}
\caption{\label{angmetfig67} The visual angle metric in the upper half plane $v_{\UH}(x,y)=\omega$.}
\end{figure}

Therefore,
\begin{eqnarray}\label{vh1}
v_{\UH}(x,y)=
\left\{\begin{array}{ll}
\arccos\frac{2\sqrt{x_2y_2}|x-y|+(x_1-y_1)(x_2+y_2)}{(x_2+y_2)|x-y|+2\sqrt{x_2y_2}(x_1-y_1)},&\,\,\, x_1\le y_1\,,x_2< y_2,\\
\arccos\frac{4x^2_2-(x_1-y_1)^2}{4x^2_2+(x_1-y_1)^2},&\,\,\,x_1\neq y_1\,,x_2=y_2.
\end{array}\right.
\end{eqnarray}
In particular, if $x_2=y_2$ and $y_1=-x_1>0$ , then
\begin{eqnarray}\label{vh2}
v_{\UH}(x,y)=2\arctan\frac{y_1}{y_2}.
\end{eqnarray}
If $x_1=y_1$, then
\begin{eqnarray}\label{vh3}
v_{\UH}(x,y)=\arccos\frac{2\sqrt{x_2 y_2}}{x_2+ y_2}.
\end{eqnarray}

\begin{lem}\label{le6}
 Let $a\in\Hn$. Let $C=S(a,r)$ be a circle centered at $a$ with radius $r$ in $\overline{\Hn}$ and tangent to $\partial\Hn$ at point $z$.
Let two distinct points $x'\,,y'\in C$ such that $|x'-z|=|y'-z|$ and $\ang(x',z,y')=\alpha\in(0,\pi)$.
Then
for arbitrary two distinct points $x\,,y\in C$ with $\ang(x,z,y)=\alpha$, there holds
\begin{eqnarray}\label{cor2dd}
d(x\,,\partial \Hn)d(y\,,\partial \Hn)\leq d(x'\,,\partial \Hn)d(y'\,,\partial \Hn).
\end{eqnarray}
\end{lem}

\begin{proof}
Without loss of generality, we may assume that $C\in\overline{\UH}$ and $z=0$. Choose two distinct points $x\,,y\in C$ and $\ang(x,z,y)=\alpha$. By symmetry, we may also assume that $|x|\leq|y|$, and the triples $(x,z,y)$ and $(x',z,y')$ are labeled in the positive order on $C$, respectively (see Figure \ref{angmetfig7}).

It is clear that $r=|a|$,\, $[x',y']\perp L(0,a)$, namely, $x'\,,y'$ are symmetry with respect to $L(0,a)$. Furthermore, inequality (\ref{cor2dd}) reduces to
$${\rm Im}\, x\, {\rm Im}\, y\leq {\rm Im}\, x'\, {\rm Im}\, y'.$$

\begin{figure}[h]
\subfigure[]{\includegraphics[width=.3 \textwidth]{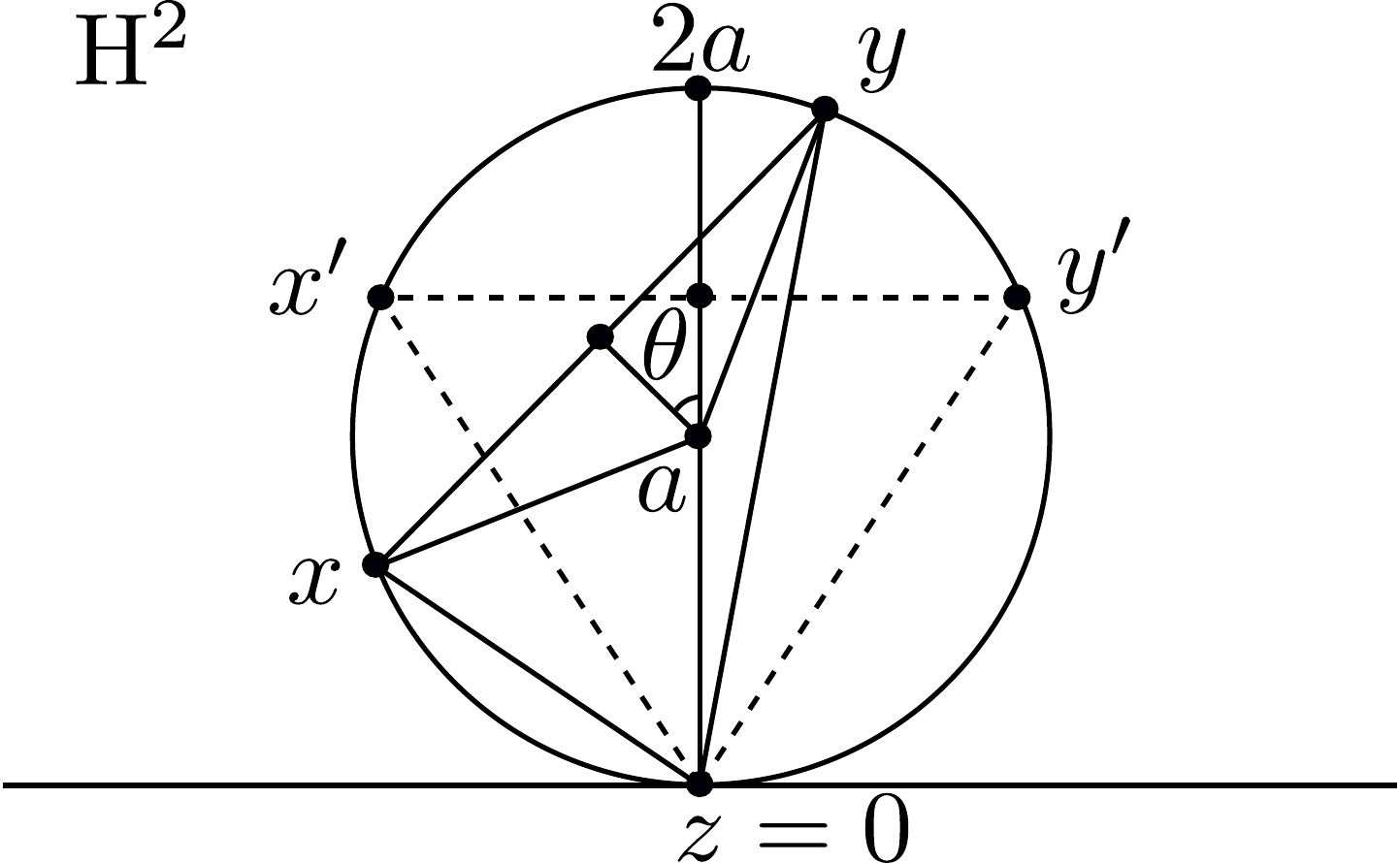}}
\hspace{.03 \textwidth}
\subfigure[]{\includegraphics[width=.3 \textwidth]{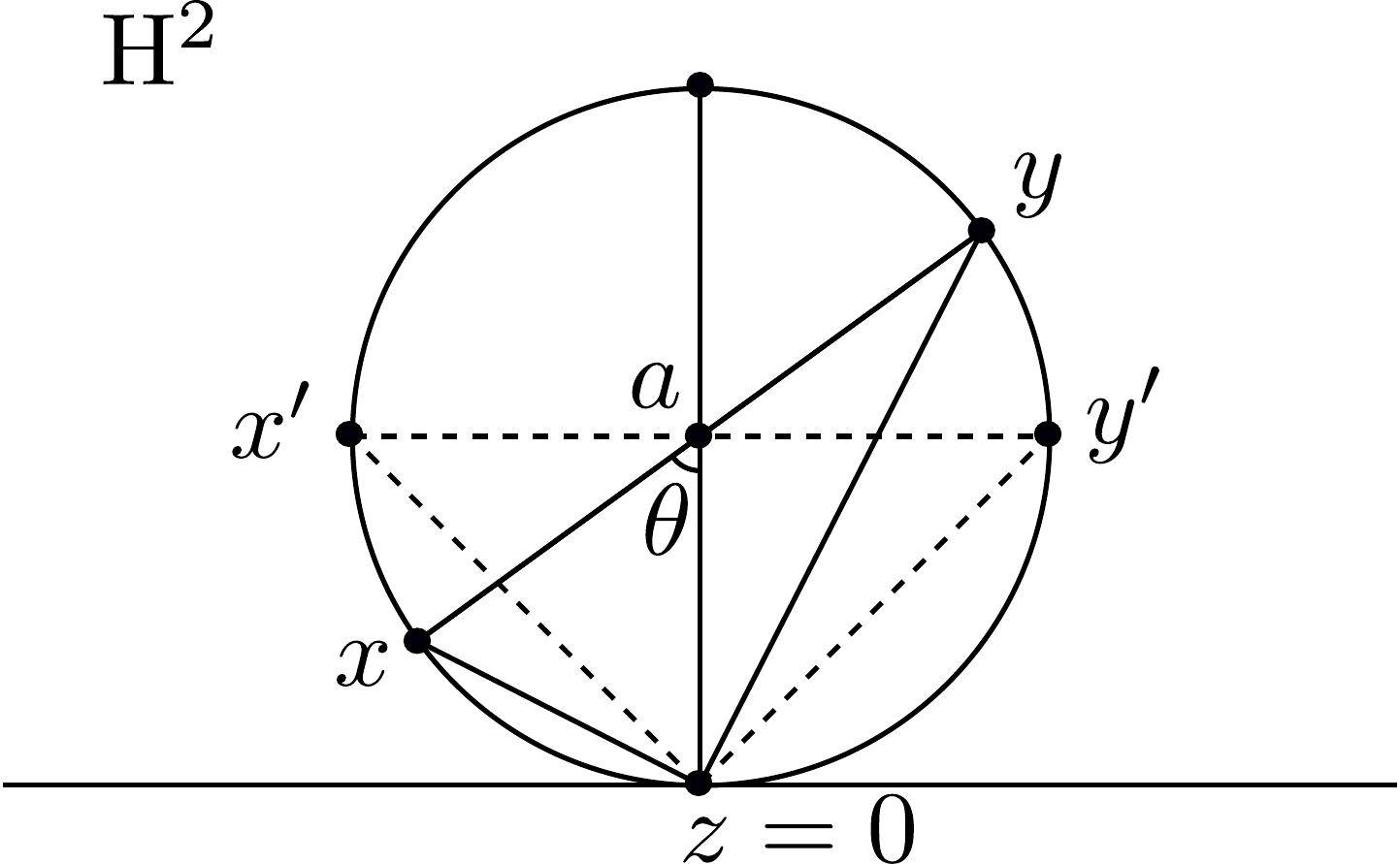}}
\hspace{.03 \textwidth}
\subfigure[]{\includegraphics[width=.3\textwidth]{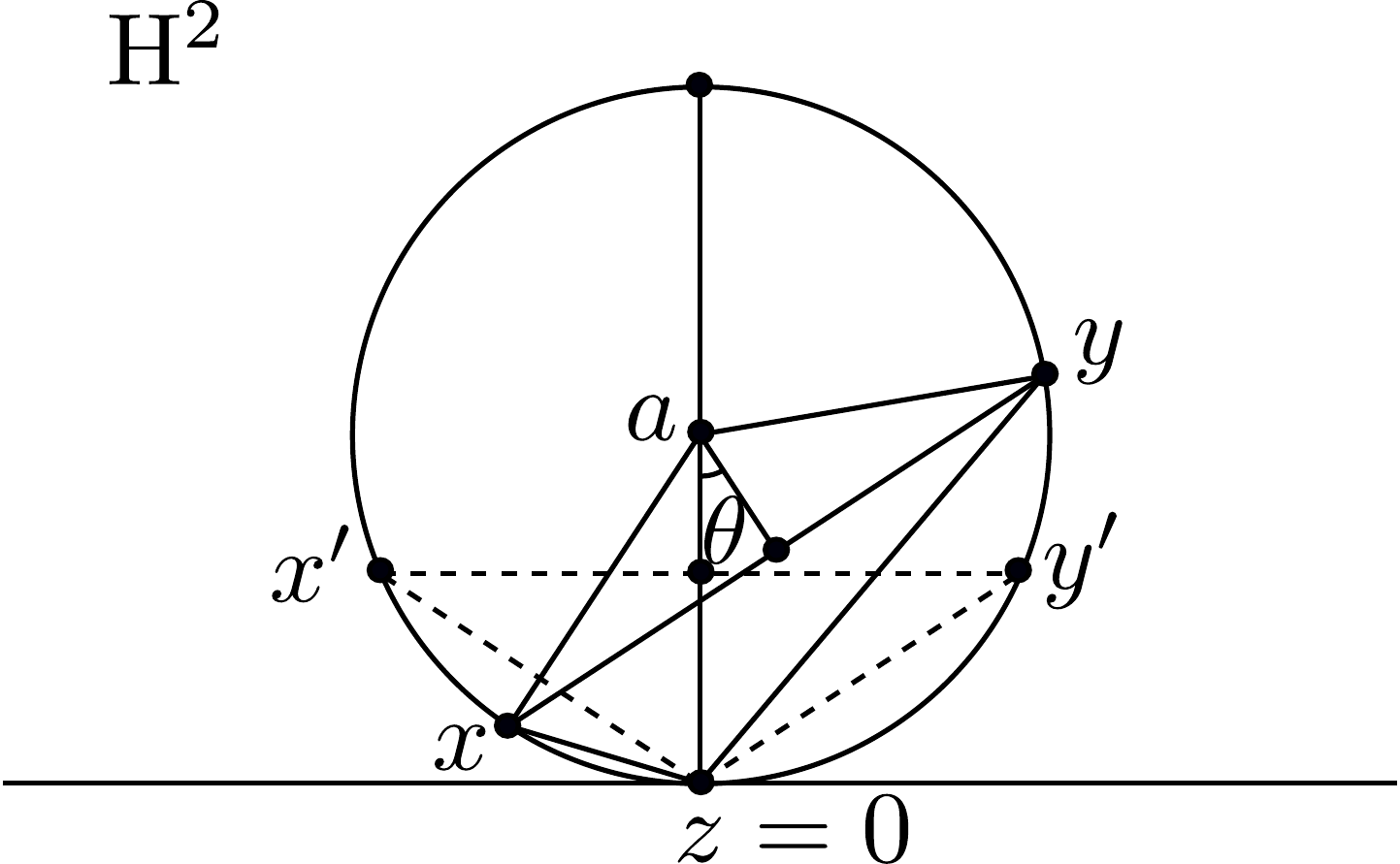}}
\caption{\label{angmetfig7} Proof of Lemma \ref{le6}. Here $\ang(x,z,y)=\ang(x',z,y')=\alpha$ and ${\rm Im}\, x\, {\rm Im}\, y\leq {\rm Im}\, x'\, {\rm Im}\, y'$.}
\end{figure}

In the same way as in the proof of Lemma \ref{le4}, we also divide the proof into three cases.
{\it Case 1.} $0<\alpha<\frac{\pi}{2}$.
Let $\theta=\ang(2a,a,\frac{x+y}{2})\in[0,\pi-\alpha)$. It is clear that $\ang(2a,a,\frac{x'+y'}{2})=0.$
Then $$x=a(1+e^{i(\alpha+\theta)})\,\,\,\,{\rm and}\,\,\,\,y=a(1+e^{-i(\alpha-\theta)}).$$
Moreover,
$${\rm Im}\, x\,{\rm Im}\, y=|a|^2 f(\theta),$$
where $ f(\theta)=(1+\cos(\alpha+\theta))(1+\cos(\alpha-\theta))$, by  Lemma \ref{at}, we have
$${\rm Im}\, x\, {\rm Im}\, y\leq{\rm Im}\, x'\, {\rm Im}\, y'.$$

{\it Case 2.} $\alpha=\frac{\pi}{2}$.
Let $\theta=\ang(0,a,x)\in(0,\pi/2]$. It is clear that $\ang(0,a,x')=\ang(0,a,y')=\frac{\pi}{2}.$
Then $$x=a(1-e^{-i\theta})\,\,\,\,{\rm and}\,\,\,\,y=a(1-e^{i(\pi-\theta)}).$$
Moreover,
$${\rm Im}\, x\,{\rm Im}\, y=|a|^2 \sin^2\theta\leq|a|^2 \sin^2\frac{\pi}{2},$$
and hence
$${\rm Im}\, x\, {\rm Im}\, y\leq {\rm Im}\, x'\, {\rm Im}\, y'.$$

{\it Case 3.} $\frac{\pi}{2}<\alpha<\pi$.
Let $\theta=\ang(0,a,\frac{x+y}{2})\in[0,\pi-\alpha)$. It is clear that $\ang(0,a,\frac{x'+y'}{2})=0.$
Then $$x=a(1-e^{-i(\pi-\alpha-\theta)})\,\,\,\,{\rm and}\,\,\,\,y=a(1-e^{i(\pi-\alpha+\theta)}).$$
Moreover,
$${\rm Im}\, x\,{\rm Im}\, y=|a|^2 f(\theta),$$
where $ f(\theta)=(1+\cos(\alpha+\theta))(1+\cos(\alpha-\theta))$, by Lemma \ref{at} we have
$${\rm Im}\, x\, {\rm Im}\, y\leq{\rm Im}\, x'\, {\rm Im}\, y'.$$

By Case 1-3, we complete the proof.
\end{proof}

\begin{cor}\label{cor4}
Let $x\,,y\,,x'\,,y'$ be as in Lemma \ref{le6}. Then
$$\rho_{\Hn}(x,y)\geq\rho_{\Hn}(x',y').$$
\end{cor}

\begin{lem} \label{le5}
Let $x\,,y\in \Hn$ and $d(x\,,\partial \Hn)=d(y\,,\partial \Hn)$. Then
$$v_{\Hn}(x,y) \le \rho_{\Hn}(x,y),$$
and the inequality is sharp.
\end{lem}
\begin{proof}
It suffices to prove for the 2-dimensional case. Let $x\,,y$ be two distinct points in $\UH$.  Since both metrics are invariant under translations, we may assume that
$y_1={\rm\,Re}\,y=-{\rm\,Re}\,x>0$ and $y_2={\rm\,Im}\,y={\rm\,Im}\,x>0$. Then
$$\rho_{\UH}(x,y)={\rm arch}\left(1+2\left(\frac{y_1}{y_2}\right)^2\right).$$
By (\ref{vh2}) and making substitution of $r=\frac{y_1}{y_2}$ in Lemma \ref{le1}(4), we have
$$v_{\UH}(x,y) \le \rho_{\UH}(x,y)$$
and the inequality is sharp if $r\rightarrow 0^+$.
\end{proof}

\begin{thm}\label{vrho3}
Let $x\,,y\in \Hn$. Then
$$v_{\Hn}(x,y) \le \rho_{\Hn}(x,y),$$
and the inequality is sharp.
\end{thm}
\begin{proof}
For $x\,,y\in \Hn$, $x\neq y$ and $d(x\,,\partial \Hn)\neq d(y\,,\partial \Hn)$, by Lemma \ref{le6} , Lemma \ref{le5} and Corollary \ref{cor4}, there exist  $x'\,,y'\in \Hn$ such that $d(x'\,,\partial \Hn)= d(y'\,,\partial \Hn)$, $|x-y|=|x'-y'|$ and
$$v_{\Hn}(x,y)=v_{\Hn}(x',y')\le \rho_{\Hn}(x',y')\leq \rho_{\Hn}(x,y).$$
Together with Lemma \ref{le5}, the inequality holds for all $x\,,y\in \Hn$ and it is sharp.
\end{proof}

\begin{conjecture}
There exists a constant $c\in(1.432, 1.433)$ such that for all $x\,,y\in \Hn$
$$v_{\Hn}(x,y) \le c j_{\Hn}(x,y).$$
 \end{conjecture}

The following lemma shows the equality case between the visual angle metric and the hyperbolic metric in the upper half space.
\begin{lem}\label{mthmhle}
Let $x\,,y\in \Hn$. Then
$$\tan v_{\Hn}(x,y)={\rm sh}\frac{\rho_{\Hn}(x,y)}{2}$$
if and only if $L(x,y-x)$ is perpendicular to the boundary $\partial \Hn$.
\end{lem}
\begin{proof}
It suffices to consider the 2-dimensional case.

For $x\,,y\in \UH$ and $x\neq y$. Let $z\in E^{\omega}_{xy}\cap\partial\UH$, where $E^{\omega}_{xy}$ is the envelope such that $\omega$ is the supremum angle in Definition \ref{maxang}, i.e., $v_{\UH}(x,y)=\ang(x,z,y)=\omega$.
Then there exists exactly one circle $S^1(a,r)$ which passes through $x\,,y\,,z$ and is tangent to $\partial\UH$. For the convenience of proof, we suppose that the three points $x\,,z\,,y$ are labeled in the positive order on $S^1(a,r)$. Without loss of generality, we may assume that $z=0$ and $|x|\le |y|$. By geometric observation, $v_{\UH}(x,y)\in(0,\pi/2)$ if $L(x,y-x)$ is perpendicular to $\partial \UH$ (cf. Figure \ref{angmetfig7}(a)) or $\tan v_{\UH}(x,y)={\rm sh}\frac{\rho_{\UH}(x,y)}{2}$. It is clear that
$$|x-y|=2|a|\sin\omega.$$
Let $\theta=\ang(2a,a,\frac{x+y}2)$. By the proof of Case 1 in Lemma \ref{le6}, we have
$$\im\,x\im\,y=|a|^2\,f(\theta),$$
where $f(\theta)$ is as in Lemma \ref{at} by taking $\alpha=\omega$.
Since
$${\rm sh}\frac{\rho_{\UH}(x,y)}{2}=\sqrt{\frac 12({\rm ch}\rho_{\UH}(x,y)-1)}=\frac{|x-y|}{2\sqrt{\im\, x \im\, y}},$$
by Lemma \ref{at}, we have
$$\tan v_{\UH}(x,y)={\rm sh}\frac{\rho_{\UH}(x,y)}{2}\Leftrightarrow \sqrt{f(\theta)}=\cos\omega\Leftrightarrow\theta=\pi/2\Leftrightarrow \re\,x=\re\,y.$$

This completes the proof.
\end{proof}
\begin{thm}\label{mthmh}
Let $x\,,y\in \Hn$. Let $\rho_{\Hn}^*(x,y)=\arctan\left({\rm sh}\frac{\rho_{\Hn}(x,y)}{2}\right)$. Then
$$\rho_{\Hn}^*(x,y)\leq v_{\Hn}(x,y)\leq 2\rho^*_{\Hn}(x,y).$$
The equality holds in the left-hand side if and only if $L(x,y-x)$ is perpendicular to the boundary $\partial \Hn$
and in the right-hand side if and only if $L(x,y-x)$ is parallel to the boundary $\partial \Hn$.
\end{thm}

\begin{proof}
It suffices to consider the 2-dimensional case.

For $x\,,y\in \UH$ and $x\neq y$. In the same way as in the proof of Lemma \ref{mthmhle}, we have the point $z$ and the angle $\omega$ such that $v_{\UH}(x,y)=\ang(x,z,y)=\omega$. We also have the circle $S^1(a,r)$ which passes through $x\,,y\,,z$ and is tangent to $\partial\UH$.
By Lemma \ref{le6}, there exist $x'\,,y'\in S^1(a,r)$ such that $\ang(x',z,y')=\ang(x,z,y)$ and $\im x'=\im y'$.
For the convenience of proof, we suppose that the three points $x\,,z\,,y$ are labeled in the positive order on $S^1(a,r)$, and so are $x'\,,z\,,y'$.
Without loss of generality, we still assume that $z=0$ and $|x|\le |y|$ (cf. Figure \ref{angmetfig7}).
By the proof of Lemma \ref{le6}, we have
$$|x'-y'|=|x-y|=2|a|\sin\omega$$
and
$$\im\,x'\im\,y'=|a|^2(1+\cos\omega)^2.$$

Therefore, by Lemma \ref{le6}
$$\tan\frac{\omega}{2}=\frac{|x'-y'|}{2\sqrt{\im\, x' \im\, y'}}\leq \frac{|x-y|}{2\sqrt{\im\, x \im\, y}}={\rm sh}\frac{\rho_{\UH}(x,y)}{2},$$
which implies the right-hand side of the inequality with equality if and only if $\im\, x=\im\, y$ by the proof of Lemma \ref{le6}.

To prove the left-hand side of the inequality, we only need to consider $v_{\UH}(x,y)\in (0,\pi/2)$ since ${\rm sh}\frac{\rho_{\Hn}(x,y)}{2}$ is always nonnegative. Let $x''\,,y''$ be two points such that $x''\,,z\,,y''$ are labeled in the positive order on $S^1(a, 1-|a|)$  and $|x''-y''|=|x-y|$, and $L(x'',y''-x'')$ is perpendicular to $\partial\UH $. Then by the definition of the visual angle metric, Lemma \ref{mthmhle} and Lemma \ref{at}, we get
$$v_{\UH}(x,y)=v_{\UH}(x'',y'')=\rho^*_{\UH}(x'',y'')\ge \rho^*_{\UH}(x,y).$$
Thus we prove the left-hand side of the inequality, and the equality holds if and only if $\re\, x=\re\, y$.

This completes the proof.
\end{proof}

\begin{prop}\label{prho}
For $G\in\{\Bn, \Hn\}$. Then $\rho^*_G(x,y)=\arctan\left({\rm sh}\frac{\rho_G(x,y)}{2}\right)$ is a M\"obius invariant metric.
\end{prop}
\begin{proof}
The function $f: x\mapsto \arctan ({\rm sh} (x/2))$ is increasing on $[0,\infty)$, $f(x)/x$ is decreasing $(0,\infty)$, and $f(0)=0$. Therefore, by Lemma \ref{fst}, $\rho^*_G(x,y)$ is a metric, and the M\"obius invariance immediately follows by the M\"obius invariance of $\rho_G(x,y)$.
\end{proof}

\begin{proof}[Proof of Theorem \ref{mthm1}]
By Theorem \ref{mthmb} and Theorem \ref{mthmh}, the results follow immediately.
\end{proof}

\end{nonsec}

\medskip

\section{Lipschitz constants under M\"obius transformations}

It is clear that the visual angle metric is similarity invariant but not  M\"obius invariant. However, by Theorem \ref{mthm1} and Proposition \ref{prho} this metric is not changed by more than a factor $2$ under the M\"obius transformations from $G$ onto $G'$ for $G\,,G'\in\{\Bn\,,\Hn\}$.
In this section, we prove the main theorems of sharp Lipschitz constants for the visual angle metric under several M\"obius transformations.

\begin{proof}[Proof of Theorem \ref{vmthm1}]
It suffices to prove the 2-dimensional case by the definition of the visual angle metric.
By Theorem \ref{mthm1} and the M\"obius invariance of $\rho^*_{\BB}(x,y)$, it is clear that
 $$v_{\BB}(x,y)/2\le v_{\BB}(f(x),(y))\le 2 v_{\BB}(x,y).$$

For the sharpness, let $a\in(0,1)$, then $T_a(z)=\frac{z-a}{1-\bar{a}z}\in \mathcal{GM}(\BB)$. Let $x=i\,t$ and $y=-i\,t\,(0<t<1)$. Then
$$T_a(x)=-\frac{a(1+t^2)-i\,t(1-a^2)}{1+a^2t^2}\,\,\,{\rm and} \,\,\,T_a(y)=-\frac{a(1+t^2)+i\,t(1-a^2)}{1+a^2t^2}.$$
Since $|x|=|y|$ and $|T_a(x)|=|T_a(y)|$, by (\ref{omega1x}) we have
\begin{eqnarray*}
\lim_{a\rightarrow 1^-}\lim_{t\rightarrow 1^-}\frac{v_{\BB}(T_a(x),T_a(y))}{v_{\BB}(x,y)}
=\lim_{a\rightarrow 1^-}\lim_{t\rightarrow 1^-}\frac{\arctan \frac{t(1+a)}{1-a t^2}}{\arctan t}
=\lim_{a\rightarrow 1^-}\frac{4}{\pi}\arctan \frac{1+a}{1-a}=2.
\end{eqnarray*}
This completes the proof of Theorem \ref{vmthm1}.
\end{proof}


\begin{conjecture}
Let $a\in\Bn$ and $f: \Bn\rightarrow \Bn=f\Bn$ be a M\"obius transformation with $f(a)=0$. Then
$$\sup_{x\neq y\in{\Bn}}\frac{v_{\Bn}(f(x),f(y))}{v_{\Bn}(x,y)}=\frac{4}{\pi}\arctan \frac{1+|a|}{1-|a|}.$$
\end{conjecture}

\medskip

\begin{proof}[Proof of Theorem \ref{vmthm2}]
By Theorem \ref{mthm1} and Proposition \ref{prho} , the inequality is clear.

Without loss of generality, we may assume that the M\"obius transformation $f$ maps $a\in \UH$ to $0$. Then $f$ is of the form
$$
f(z)=e^{i\alpha} \frac{z-a}{z-\bar{a}}
$$
and hence
$$
f^{-1}(z)=\frac{a-\bar{a}e^{-i\alpha}z}{1-e^{-i\alpha}z},
$$
where $\alpha$ is a real constant.
Since the visual angle metric is invariant under translations, strecthings of $\UH$ onto itself and rotations of $\BB$ onto itself, we may assume that $a=i$ and $\alpha=0$.
Then we have
$$
f(z)=\frac{z-i}{z+i}\,\,\,\,\,\,\,\,{\rm and} \,\,\,\,\,\,\,\
f^{-1}(z)=i\frac{1+z}{1-z}.
$$

For the sharpness of the upper bound, let $x=-\frac{2 t}{\sqrt{1-t^2}}+i$ and $y=i\frac{1+t}{1-t}\,(0<t<1)$. Then
$$f(x)=t^2+i\,t\sqrt{1-t^2}\,\,\,{\rm and} \,\,\,f(y)=t.$$
It is easy to see that $|f(x)|=|f(y)|=t$ and $f(x)\in S^1(1/2,1/2)$. Hence $\cos(\frac12\ang(f(x),0,f(y)))=\sqrt{(1+t)/2}$ and $\sin(\frac12\ang(f(x),0,f(y)))=\sqrt{(1-t)/2}$.
By (\ref{omega1x}), we have
\begin{eqnarray}\label{bf}
& &\lim_{t\rightarrow 1^-} v_{\BB}(f(x),f(y))=\lim_{t\rightarrow 1^-}2\arctan\frac{t\sqrt{1-t}}{\sqrt2-t\sqrt{1+t}}\nonumber\\
&=& 2\arctan\lim_{t\rightarrow 1^-}\frac{\sqrt{1+t}}{\sqrt{1-t}(3t+2)}=\pi.
\end{eqnarray}

By (\ref{vh1}), we have
\be\label{u}
\lim_{t\rightarrow 1^-} v_{\UH}(x,y)=\lim_{t\rightarrow 1^-}\arccos\frac{\sqrt2\sqrt{1-t^2}-\sqrt{1-t}}{\sqrt2-\sqrt{1-t^2}\sqrt{1-t}}
=\frac{\pi}2.
\ee

Therefore, by (\ref{bf}) and (\ref{u}), we get the upper bound.

For the sharpness of the lower bound, let $x=0$ and $y=\frac{t^2}{t^2+4}-i\frac{2t}{t^2+4}\,(t>0)$. Then
$$f^{-1}(x)=i\,\,\,{\rm and} \,\,\,f^{-1}(y)=t+i.$$
By (\ref{omega0x}) and (\ref{vh1}), we have
$$
v_{\UH}(f^{-1}(x),f^{-1}(y))=\arccos\frac{4-t^2}{4+t^2}\,\,\,\,{\rm and}\,\,\,\,v_{\BB}(x,y)=\arcsin\frac{t}{\sqrt{t^2+4}}.
$$
Since
$$
\cos(v_{\UH}(f^{-1}(x),f^{-1}(y)))=1-2\sin^2(v_{\BB}(x,y)),
$$
we get
$$v_{\UH}(f^{-1}(x),f^{-1}(y))=2v_{\BB}(x,y).$$

This completes the proof of Theorem \ref{vmthm2}.
\end{proof}

\begin{proof}[Proof of Theorem \ref{vmthm3}]
By Theorem \ref{mthm1} and the M\"obius invariance of $\rho^*_{\UH}(x,y)$, it is clear that
 $$v_{\UH}(x,y)/2\le v_{\UH}(f(x),(y))\le 2 v_{\UH}(x,y).$$

For the sharpness, we devide the proof into two cases.

{\it Case 1.} $c \neq 0$ and $d \neq 0$.
Let $x=i$ and $y=i\frac{d^2}{c^2}$. Then
$$f(x)=\frac{ac+bd}{c^2+d^2}+i\frac{1}{c^2+d^2} \,\,\,{\rm and} \,\,\,f(y)=\frac{ad^3+bc^3}{cd^3+c^3d}+i\frac{1}{c^2+d^2}.$$
Since $\re\,x=\re\,y=0$ and $\im\,f(x)=\im\,f(y)=\frac{1}{c^2+d^2}$, by Theorem \ref{mthmh} and Proposition \ref{prho}, we have
$$
\frac{v_{\UH}(f(x),f(y))}{v_{\UH}(x,y)}=\frac{2\rho^*_{\UH}(f(x),f(y))}{\rho^*_{\UH}(x,y)}=2.
$$

{\it Case 2.} $c \neq 0$ and $d=0$.
Then $bc=-1$ and $f(z)=-\frac{b^2}z-ab$. It suffices to consider the map $f(z)=-\frac 1z$ since the visual angle metric is invariant under translations and stretchings from the upper half plane onto itself.

Let $x=t e^{i(\pi-t)}$ and $y=i\frac{t}{\sin t}\,(0<t<\pi/2)$. Then
$$f(x)=\frac{\cos t}{t}+i\frac{\sin t}{t} \,\,\,{\rm and} \,\,\,f(y)=i\frac{\sin t}{t}.$$
Since $\im\,f(x)=\im\,f(y)$, by (\ref{vh1}) we have
\begin{eqnarray}\label{hf}
\lim_{t\rightarrow 0^+} v_{\UH}(f(x),f(y))=\lim_{t\rightarrow 0^+}\arccos\frac{4\sin^2t-\cos^2t}{4\sin^2t+\cos^2t}=\pi
\end{eqnarray}
and by (\ref{vh1})
\be\label{uh}
\lim_{t\rightarrow 0^+} v_{\UH}(x,y)=\lim_{t\rightarrow 0^+}\arccos\sin t=\frac{\pi}2.
\ee

Therefore, by (\ref{hf}) and (\ref{uh}), we get
$$\lim_{t\rightarrow 0^+} \frac{v_{\UH}(f(x),f(y))}{ v_{\UH}(x,y)}=2.$$

This completes the proof of Theorem \ref{vmthm3}.
\end{proof}

\begin{rem}
If $c=0$ in Theorem \ref{vmthm3}, then $f(z)=a^2z+ab$. Therefore, it is clear that the Lipschitz constant under $f$ for the visual angle metric is always $1$.
\end{rem}

\bigskip

\subsection*{Acknowledgments}
This research was supported by the Academy of Finland,
Project 2600066611.


\end{document}